\documentclass[11pt]{amsart}
 
 \usepackage{upgreek}
 \usepackage{mathpazo}

\usepackage{adjustbox}
\usepackage{amsmath, amsxtra, amsthm, amsfonts,amssymb,mathtools,bm,ytableau}
\usepackage{mathrsfs}
\usepackage[mathscr]{euscript}
\usepackage{graphicx}
\usepackage{url}
\usepackage{color}
\usepackage{bbm} 
\usepackage{tikz-cd}
\usepackage{tikz}

\usepackage[T1]{fontenc}

\usepackage{float}
\floatstyle{plaintop}
\restylefloat{table}

\usepackage{enumitem}
\usepackage{comment}
\usepackage[pdftex,hidelinks]{hyperref} 
\hypersetup{
    colorlinks,
    citecolor=magenta,
    filecolor=magenta,
    linkcolor=blue,
    urlcolor=black
}
\usepackage{cleveref}

\renewcommand{\emptyset}{\varnothing}

\newcommand{\QQ}{\mathbb Q}

\newcommand{\ZZ}{\mathbb Z}

\newcommand{\id}{\mathrm{id}}

\theoremstyle{definition}
\newtheorem{thm}{Theorem}[section]
\newtheorem{cor}[thm]{Corollary}
\newtheorem{lem}[thm]{Lemma}
\newtheorem{prop}[thm]{Proposition}
\newtheorem{defn}[thm]{Definition}

\newtheorem{eg}[thm]{Example}
\newtheorem{rem}[thm]{Remark}

\numberwithin{equation}{section}

\newcommand{\indexedforests}[1][]{ % set of indexed forests
{\ifx&#1&%
    \operatorname{\mathsf{For}}
\else
    \operatorname{\mathsf{For}}^{#1}
\fi}
}
\newcommand{\forestpoly}[2][]{{\mathfrak{P}_{#2}^{\underline{#1}}}} %forest polynomials

\newcommand{\qsym}[2][]{
{\ifx&#1&%
  {\operatorname{QSym}_{#2}}
\else
  {{}^{#1}\!\operatorname{QSym}_{#2}}
\fi}
} % ring of quasisymmetric polynomials

\newcommand{\qsymide}[2][]{
{\ifx&#1&%
  {\operatorname{QSym}_{#2}^+}
\else
  {{}^{#1}\!\operatorname{QSym}_{#2}^+}
\fi}
} % ideal of positive degree quasis
\newcommand{\lcode}[1]{\operatorname{lcode}(#1)} % lehmer code
\newcommand{\supp}{\operatorname{supp}} % support
\newcommand{\compatible}[2][]{
{\ifx&#1&%
  {\mathcal{C}(#2)}
\else
  {\mathcal{C}^{m}(#2)}
\fi
}}

\newcommand{\internal}[1]{\operatorname{IN}(#1)} % set of internal nodes
\newcommand{\suchthat}{\;|\;}
\newcommand{\nvect}{\mathsf{Codes}}

\newcommand{\poly}{\operatorname{Pol}} % polynomial ring
\newcommand{\schub}[1]{\mathfrak{S}_{#1}} % schubert polynomials
\newcommand{\red}[1]{\operatorname{Red}(#1)} % reduced words
\newcommand{\des}[1]{\operatorname{Des}(#1)} % descent set
\date{}
\newcommand{\idem}{\operatorname{id}} %identity

\newcommand{\slide}[2][]{
{\ifx&#1&%
  {\mathfrak{F}_{#2}}
\else
  {\mathfrak{F}_{#2}^{\underline{#1}}}
\fi}
} % slide polynomial

\newcommand{\ct}{\operatorname{ev}_{0}} % constant term
\newcommand{\qdes}[2][]{\operatorname{LTer}_{#1}(#2)} % descent set
\newcommand{\tope}[2][]{
{\ifx&#1&%
  {\mathsf{T}_{#2}}
\else
  {\mathsf{T}_{#2}^{\underline{#1}}}
\fi}
} % trimming operator
\newcommand{\builder}[2][]{
{\ifx&#1&%
  {\mathsf{B}_{#2}}
\else
  {\mathsf{B}_{#2}^{\underline{#1}}}
\fi}
} % blossoming operator
\newcommand{\rope}[1]{\mathsf{R}_{#1}} % R-ing operator
\newcommand{\dope}[2][]{
{\ifx&#1&%
  {\mathsf{D}_{#2}}
\else
  {\mathsf{D}_{#2}^{\underline{#1}}}
\fi}
} % R-ing operator

\newcommand{\zope}{\mathsf{Z}} % R-ing operator
\newcommand{\mzope}[1]{\zope^{\underline{#1}}}

\newcommand{\End}{\operatorname{End}} % Endomorphism
\newcommand{\Trim}[1]{\operatorname{Trim}({#1})} % trimming sequences
\newcommand{\Th}[1][]{\mathsf{ThMon}^{\underline{#1}}} % Thompson % Thompson
\newcommand{\sfx}{\mathsf{x}} %x but when used as a monomial

\newcommand{\sfc}{\mathsf{c}}

\newcommand{\qseq}[2][]{
{\ifx&#1&%
  {\operatorname{QSeq}_{#2}}
\else
  {{}^{#1}\!\operatorname{QSeq}_{#2}}
\fi}
}

\newcommand{\fac}[1]{\operatorname{Fac}(#1)}
\newcommand{\last}[1]{\operatorname{Last}(#1)}

\newcommand{\pd}[1]{\operatorname{PD}(#1)} % pipe dreams
\newcommand{\diagrams}[1]{\operatorname{Diag}(#1)} % generic name for "forest diagram"

\newcommand{\ddpair}{divided difference pair }
\newcommand{\newddpair}{dd-pair }
\newcommand{\winc}{\mathsf{WInc}} %winc monoid

\title{Schubert polynomial expansions revisited}

\author{Philippe Nadeau}
\address{Universite Claude Bernard Lyon 1, CNRS, Ecole Centrale de Lyon, INSA Lyon, Université Jean Monnet, ICJ UMR5208, 69622 Villeurbanne, France}
\email{\href{mailto:nadeau@math.univ-lyon1.fr}{nadeau@math.univ-lyon1.fr}}

\author{Hunter Spink}
\address{Department of Mathematics,
University of Toronto, Toronto, ON M5S 2E4, Canada}
\email{\href{mailto:hunter.spink@utoronto.ca}{hunter.spink@utoronto.ca}}

\author{Vasu Tewari}
\address{Department of Mathematical and Computational Sciences, University of Toronto Mississauga, Mississauga, ON L5L 1C6, Canada}
\email{\href{mailto:vasu.tewari@utoronto.ca}{vasu.tewari@utoronto.ca}}

\thanks{
PN was partially supported by French ANR grant ANR-19-CE48-0011 (COMBIN\'E). HS and VT acknowledge the support of the Natural Sciences and Engineering Research Council of Canada (NSERC), respectively [RGPIN-2024-04181] and [RGPIN-2024-05433].}

\keywords{Divided differences, forest polynomials, pipe dreams, Schubert polynomials, slide polynomials}

\begin{document}

\begin{abstract}
We give an elementary approach utilizing only the divided difference formalism for obtaining expansions of Schubert polynomials  that are manifestly nonnegative, by studying solutions to the equation $\sum Y_i\partial_i=\idem$ on polynomials with no constant term. This in particular recovers the pipe dream and slide polynomial expansions.
We also show that slide polynomials satisfy an analogue of the divided difference formalisms for Schubert polynomials and forest polynomials, which gives a simple method for extracting the coefficients of slide polynomials in the slide polynomial decomposition of an arbitrary polynomial.
\end{abstract}

\maketitle

%%%%%%%%%%%%%%%%%%%%%%%%%%%%%%%%%%%%%%%%%%%%%
\section{Introduction}
%%%%%%%%%%%%%%%%%%%%%%%%%%%%%%%%%%%%%%%%%%%%%

Let $S_{\infty}$ denote the set of permutations of $\{1,2,\ldots\}$ with finite support, and let $\ell(w)$ denote the length of a permutation, the length of the smallest word in the simple transpositions $s_i=(i,i+1)$ which equals $w$. The nil-Coxeter monoid is the right-cancelative partial monoid whose elements are permutations in $S_{\infty}$, equipped with the partial monoid structure
\begin{align}\label{eqn:Sinftypartial}u\circ v=\begin{cases}uv&\text{if }\ell(u)+\ell(v)=\ell(uv)\\\text{undefined}&\text{otherwise.}\end{cases}\end{align}
There is a permutation $w/i$ such that $w=(w/i)\circ i$ if and only if $i$ is in the descent set $\des{w}=\{j\suchthat w(j)>w(j+1)\}$, in which case it is unique and given by the formula $w/i=ws_i$.
An important representation of the nil-Coxeter monoid is the divided difference representation on integral polynomials, which sends $s_i$ to the $i$'th divided difference operator $\partial_i$ given by the formula
\begin{align}
\label{eqn:partiali}\partial_i(f)=\frac{f-f(x_1,\ldots,x_{i-1},x_{i+1},x_i,\ldots)}{x_i-x_{i+1}}.\end{align}
The Schubert polynomials $\{\schub{w}\suchthat w\in S_{\infty}\}$ of Lascoux--Sch\"utzenberger \cite{LS82,LS87} are a family of polynomials indexed by permutations $w$ in $S_{\infty}$, characterized by the normalization condition $\schub{\idem}=1$, and the relations $$\partial_i\schub{w}=\begin{cases}\schub{w/i}&\text{if }i\in \des{w}\\0&\text{otherwise.}\end{cases}$$ 
Despite their relatively simple definition, Schubert polynomials are complicated combinatorial objects. 
Many combinatorial formulas for Schubert polynomials exist, such as the algorithmic method of Kohnert~\cite{As22, Kohnert1991}
, the pipe dreams of Bergeron--Billey~\cite{BerBil93} and Fomin--Kirillov \cite{FK96}, the slide expansions of Billey--Jockusch--Stanley~\cite{BJS93} and Assaf--Searles \cite{AssSea17}, the balanced tableaux of Fomin--Greene--Reiner--Shimozono~\cite{FoGrReSh1997balanced}, the bumpless pipe dreams of Lam--Lee--Shimozono~\cite{LamLeeShi21}, and the prism tableau model of Weigandt--Yong~\cite{WY18}.

Expansions of Schubert polynomials have been almost exclusively studied from a ``top-down'' perspective -- for $w_{0,n}$ the longest permutation in $S_n$ one checks the conjectured formula agrees with the Ansatz $\schub{w_{0,n}}=x_1^{n-1}x_2^{n-2}\cdots x_{n-1}$, and then verifies the conjectured formula transforms correctly under applications of $\partial_i$. 
It seems the approaches to studying Schubert formulae that are ``bottom-up''  are rather limited. 
They fall into a broad class of results revolving around Pieri rules \cite{Sot96} (containing Monk's rule \cite{Mo59} as a special case)  expanding the product of $\schub{w}$ with elementary and complete homogenous symmetric polynomials  via the $k$-Bruhat order \cite{BS98} to establish relations between Schubert polynomials related by nonadjacent transpositions \cite[\S 3]{LS85}.
Another approach, relying on the geometry of Bott--Samelson varieties, is due to Magyar \cite{Mag98} and it builds Schubert polynomials by interspersing isobaric divided differences with multiplications by terms of the form $x_1\cdots x_i$ (cf. \cite{MSStD22} for a generalization to Grothendieck polynomials using  combinatorial tools).

In this paper we develop a new general method for finding combinatorial expansions of Schubert polynomials which works from the bottom-up, by directly reconstructing a Schubert polynomial $\schub{w}$ from the collection of Schubert polynomials $\schub{ws_i}$ for $i\in \des{w}$.

We demonstrate here our technique on a simpler toy example, where we recover the family of normalized monomials $\{S_c=\frac{\sfx^c}{c!}\coloneqq \frac{x_1^{c_1}\cdots x_\ell^{c_\ell}}{c_1!\cdots c_\ell!}\suchthat c=(c_1, \dots, c_\ell)\}$ using only the indirect information that they are homogenous with $S_\emptyset=1$ and satisfy
\begin{align}\label{eqn:ddxdualrecursion}
\frac{d}{dx_i}S_c=\begin{cases}S_{c-e_i}&\text{if }c_i\ge 1\\0&\text{otherwise.}\end{cases}
\end{align}
where $c-e_i=(c_1,\ldots,c_{i-1},c_i-1,c_{i+1},\ldots,c_\ell)$.
Our technique is motivated by the Euler's famous theorem
$$\sum_{i=1}^\infty x_i\frac{d}{dx_i}f=kf$$
for $f$ a homogenous polynomial of positive degree $k$. Iteratively applying this identity shows that
$$\sum_{i_1,\ldots,i_k}x_{i_1}\cdots x_{i_k}\frac{d}{dx_{i_1}}\cdots \frac{d}{dx_{i_k}}f=k!\idem$$
on homogenous polynomials of degree $k$, and grouping together terms with the same derivatives applied to $f$ shows that
$$\sum_{(c_1,\dots,c_\ell)}\frac{\sfx^c}{c!}\left(\frac{d}{dx_1}\right)^{c_1}\cdots \left(\frac{d}{dx_k}\right)^{c_\ell}=\mathrm{id}.$$
Applying this identity to $S_c$ shows that $S_c=\frac{\sfx^c}{c!}$ as desired. 
Notably this calculation does not use the Ansatz that the family of polynomials we are seeking are monomials.

Let $\poly\coloneqq \ZZ[x_1,x_2,\dots]$, and let $\poly^+\subset \poly$ denote the ideal of polynomials with no constant term.
Our method relies on finding degree $1$ ``creation operators'' $Y_1,Y_2,\ldots$ that solve the equation
$$\sum_{i=1}^{\infty} Y_i\partial_i=\idem$$
on $\poly^+$. 
Applying this equation to a Schubert polynomial and recursing gives an expansion
$$\sum_{(i_1,\ldots,i_k)\in \red{w}}Y_{i_k}\cdots Y_{i_1}(1)=\schub{w},$$
where $\red{w}$ is the set of reduced words for $w$. In particular, if each $Y_i$ is a monomial nonnegative operator then this produces a monomial nonnegative expansion of $\schub{w}$. 
Given the simplicity, we now show that Schubert polynomials have a nonnegative monomial  expansion using this technique by producing one such family of creation operators (this later appears as \Cref{subsec:partialcreation}, we will produce an additional family in \Cref{subsec:Di_applications}). 
Define the map
$$\rope{i}(f)=f(x_1,\ldots,x_{i-1},0,x_i,x_{i+1},\ldots).$$
Then $$\idem=\rope{1}+(\rope{2}-\rope{1})+(\rope{3}-\rope{2})=\rope{1}+\sum_{i=1}^{\infty}x_i\rope{i}\partial_i.$$
Moving $\rope{1}$ to the other side and noting that $\id-\rope{1}$ is invertible on polynomials with no constant term with inverse $\zope=\idem+\rope{1}+\rope{1}^2+\cdots$, we conclude that
$$\sum \zope x_i\rope{i}\partial_i=\mathrm{id}.$$
Applying this to $\schub{w}$ immediately gives the following.
\begin{thm}[Corollary~\ref{cor:ZxRcreation}]\label{th:main_1}
We have the following monomial positive expansion
    $$\schub{w}=\sum_{(i_1,\ldots,i_k)\in \red{w}}\zope x_{i_k}\rope{i_k}\cdots \zope x_{i_1}\rope{i_1}(1).$$
\end{thm}

We generalize these ideas to a more general situation $(X,M)$ we call a ``\ddpair'' (\newddpair henceforth), in which the compositions of degree $-1$ polynomial endomorphisms $X_1,X_2,\ldots$, given by ``shifts'' of a fixed endomorphism $X$, form a representation of a right-cancelable partial graded monoid $M$ generated in degree $1$. 
Writing $\last{w}$ for the analogue of the descent set of $w$, we will say that a family of polynomials $\{S_w\suchthat w\in M\}$ is ``dual'' to the \newddpair if it satisfies the normalization condition $S_1=1$ and
    $$X_iS_w=\begin{cases}S_{w/i}&\text{ if }i\in \last{w}\\0&\text{otherwise.}\end{cases}$$
It is then natural to ask the following.
\begin{enumerate}
    \item Assuming there is such a family of polynomials $\{S_w\suchthat w\in M\}$, can we write down a formula for $S_w$?
    \item Does such a family of polynomials exist in the first place?
\end{enumerate}
These questions came up naturally from our previous paper \cite{NST_1} for the operators $$\tope[m]{i}(f)=\frac{f(x_1,\ldots,x_{i-1},x_i,0^m,x_{i+1},\ldots)-f(x_1,\ldots,x_{i-1},0^m,x_i,x_{i+1},\ldots)}{x_i}$$ called ``$m$-quasisymmetric divided difference operators''. There we had to essentially guess (via computer assistance) a formula for the family of $m$-forest polynomials, and then through a tedious and unenlightening computation \cite[Appendix]{NST_1} show that they interact in the expected way with the $\tope[m]{i}$ operators.

The analogue of creation operators $Y_i$ such that $\sum Y_iX_i=\idem$ on polynomials with no constant term can be used to solve the first question analogously as for Schubert polynomials, and we find such operators for $m$-forest polynomials without difficulty.

For the second question, we show surprisingly that if a \newddpair has creation operators, then the only additional thing that is needed to ensure that the dual family of polynomials exist is a ``code map'' $c:M\to \nvect$ from the partial monoid to finitely supported sequences of nonnegative integers, so that the highest index of a nonzero element of $c(m)$ is the maximal element of $\last{w}$. 
The Lehmer code of permutations works for the $\partial_i$ formalism, while the $m$-Dyck path forest code \cite[Definition 3.5]{NST_1} works for the $\tope[m]{i}$ formalism: this shows directly that Schubert polynomials and $m$-forest polynomials exist without any Ansatz or combinatorial model.

As a further application, we study the well-known family of polynomials called ``slide polynomials''  investigated in detail by Assaf--Searles \cite{AssSea17}; this family is also present in earlier works \cite{BJS93,Hi00} (see \cite{hicks2024quasisymmetric} for more on the relation to Hivert's foundational work).
Forest polynomials and Schubert polynomials decompose nonnegatively in terms of this family (see respectively \cite{NT_forest} and \cite{AssSea17,BJS93}). 
A slide polynomial is determined by a sequence of positive integers $(a_1,a_2,\ldots,a_k)$, and the distinct slide polynomials $\slide{a_1,\dots,a_k}$ are indexed by weakly increasing sequences $1\leq a_1\le a_{2}\le \cdots \le a_k$.
We construct a \newddpair for the operators $$\dope{i}(f)=\frac{f(x_1,\ldots,x_{i-1},x_i,0,0,\ldots)-f(x_1,\ldots,x_{i-1},0,x_i,0,\ldots)}{x_i}$$ whose compositions are governed by the partial monoid whose only relations are that $\dope{i}\dope{j}$ is undefined for $i>j$, such that the slide polynomials form the dual family of polynomials. 
This gives a fast and practical method for directly extracting coefficients of an arbitrary polynomial in the slide basis. 
Since fundamental quasisymmetric polynomials are a subfamily within slide polynomials, this generalizes \cite[Corollary 8.6]{NST_1}. 

\begin{thm}[\Cref{cor:slide_expansion_easy}]
The slide expansion of a degree $k$ homogenous polynomial $f\in \poly$ is given by
     \[
        f=\sum_{1\le i_1\le \cdots \le i_k} (\dope{i_1}\cdots\dope{i_k}f) \, \slide{i_1,\dots,i_k}.
     \]
\end{thm}

Associated to the $\dope{i}$ are a new family of operators we call ``slide creators'' $\builder{i}$ that have the property that for any sequence $a_1,\ldots,a_k$ (not necessarily weakly increasing) we have
$$\slide{a_1,\ldots,a_k}=\builder{a_k}\cdots \builder{a_1}(1),$$
and
$$\sum \builder{i}\partial_i=\sum \builder{i} \tope{i}=\sum \builder{i}\dope{i}=\idem$$
on $\poly^+$,
i.e. they function as creation operators for Schubert polynomials, forest polynomials, and slide polynomials themselves simultaneously. 
Using these facts, we obtain the known slide polynomial expansions of Schubert and forest polynomials.

\begin{table}[!h]
    \centering
    \begin{adjustbox}{max width=\textwidth}
    \renewcommand{\arraystretch}{1.2}
    \begin{tabular}{|c|c|c|c|c|}
    \hline
        $\mathsection$&\textbf{Monoid}&\textbf{Divided differences}&\textbf{Dual polynomials}&\textbf{Creation operators}\\
        \hline
        \ref{sec:schubert}&Nil-Coxeter monoid $S_{\infty}$&$\partial_i$&Schubert polynomials $\schub{w}$&$\zope x_i\rope{i}$ and ($\mathsection$\ref{sec:Slide})$\builder{i}$\\
        \hline \ref{sec:Forest}&Thompson monoid $\Th[]$&$\tope{i}=\rope{i}\partial_i=\rope{i+1}\partial_i$&Forest polynomials $\forestpoly{F}$&$\zope x_i$ and ($\mathsection$\ref{sec:Slide}) $\builder{i}$\\    
        &$m$-Thompson monoid $\Th[m]$&$\tope[m]{i}=\tope{i}\rope{i+1}^{m-1}$&$m$-forest polynomials $\forestpoly[m]{F}$&$\mzope{m}x_i$ and ($\mathsection$\ref{sec:Slide}) $\builder[m]{i}$\\
        \hline \ref{sec:Slide}&Weakly increasing monoid $\winc$&$\dope{i}=\rope{i+1}^{\infty}\partial_i=\rope{i+1}^{\infty}\tope{i}$&Slide polynomials $\slide{\textbf{i}}$&$\builder{i}$\\
        &&$\dope[m]{i}=\rope{i+1}^{\infty}\tope[m]{i}$&$m$-slide polynomials $\slide[m]{\textbf{i}}$&$\builder[m]{i}$\\
        &&$\dope[\infty]{i}=\rope{i+1}^{\infty}\tope[\infty]{i}=\tope[\infty]{i}$&Monomials $x_{\textbf{i}}$&$\builder[\infty]{i}$\\
        \hline
    \end{tabular}
    \renewcommand{\arraystretch}{1}
    \end{adjustbox}
    \caption{Divided difference formalisms}    
\label{table:DDFormalisms}
\end{table}

\subsection{Outline of the paper}
See \Cref{table:DDFormalisms} for an overview of where we address each family of polynomials we consider in the paper. In \Cref{sec:ddform} we set up the notion of divided difference pairs, and study  creation operators and code maps. In \Cref{sec:schubert} we study Schubert polynomials. In \Cref{sec:Forest} we study forest polynomials, including $m$-forest polynomials. In \Cref{sec:Slide} we study slide polynomials and $m$-slide polynomials, which include monomials as a limiting case.

%%%%%%%%%%%%%%
\subsection*{Acknowledgements}
We are very grateful to Dave Anderson, Sara Billey, Igor Pak, Greta Panova, Brendan Pawlowski, Richard Stanley, and Josh Swanson for enlightening discussions.
%%%%%%%%%%%%%

%%%%%%%%%%%%%%%%%%%%%%%%%%%%%%%%%%%%%%%%%%%%%%%%%%%%%%%%%%
\section{Divided differences and creation operators}
\label{sec:ddform}
%%%%%%%%%%%%%%%%%%%%%%%%%%%%%%%%%%%%%%%%%%%%%%%%%%%%%%%%%%%%

We describe a general framework which encodes the duality between $\partial_i$ and $\schub{w}$. In our framework the pair $(\partial,S_{\infty})$ will be called a \ddpair (\newddpair for short), and $\{\schub{w}\suchthat w\in S_{\infty}\}$ will be called a ``dual family of polynomials'' to this \newddpair. The two main mathematical insights are as follows.  \begin{enumerate}
    \item The existence of certain ``creation operators''  lead to explicit formulas for the dual polynomials, assuming the dual family of polynomials exist.
    \item Creation operators together with a ``code map'' shows that the dual polynomials exist,  without needing to verify any particular Ansatz or combinatorial model that  interacts well with the operators.
\end{enumerate}
These considerations are new and interesting even in the case of Schubert polynomials. For example, because we have the $\zope x\rope{}$ creation operators mentioned in the introduction, we will see in \Cref{sec:schubert} that the existence of the Lehmer code on permutations immediately implies that Schubert polynomials exist without any Ansatz or direct verification that the $\zope x\rope{}$ recursion interacts well with the $\partial_i$ operators.
In later sections we will apply this formalism to other families of polynomials.%, -- the reader might want to read \Cref{sec:schubert} in parallel with this one.

\begin{rem}
\label{rem:ZZorQQ}
The operators and families of polynomials of interest to us in this paper have integer coefficients, so we will set everything up over $\ZZ$. This will exclude certain parts of the $\frac{d}{dx_i}$ example from the introduction because of the denominators present in the normalized monomials $S_c=\frac{\sfx^c}{c!}$. However, all of the theorems we have work equally well over $\QQ$, and we will indicate through this section how such modifications apply to this particular example. 
\end{rem}

\subsection{Partial monoids and polynomial representations}
We start by recalling some notions on partial monoids: these will encode the combinatorics of relations between families of operators.

A \textit{partial monoid} $M$ is a set equipped with a partial product map $M\times M \dashrightarrow M$ denoted by concatenation, together with a unit $1$, such that $1m=m1=m$ for all $m\in M$, and $m(m'm'')=(mm')m''$ for any $m,m',m''$, in the sense that either both products are undefined, or both are defined and equal.

\begin{rem} We have a monoid when the map is total, that is when products are always defined. Given a partial monoid $M$, one forms a monoid on the one-element extension $M\sqcup \{\mathbf{0}\}$ by setting $mm'=\mathbf{0}$ when the product is undefined in $M$, and if $m$ or $m'$ is $\mathbf{0}$. The notions of partial monoids and monoids with zero are thus essentially equivalent. 
\end{rem}

A \textit{polynomial representation} of $M$ is a map $\Phi:M\to \End(\poly)$  assigning an endomorphism of $\poly$ to each element of $M$ such that $\Phi(1)=\idem$, and such that for  $u,v\in M$ we have
$$\Phi(u)\Phi(v)=\begin{cases}\Phi(uv)&\text{if }uv\text{ is defined}\\0&\text{otherwise.}\end{cases}$$

A partial monoid $M$ is \textit{graded} if there is a length function $\ell:M\to \{0,1,2,\ldots\}$ such that $\ell(uv)=\ell(u)+\ell(v)$ whenever $uv$ is defined. We write $M_k\subset M$ for those elements of degree $k$. 
We always have $M_0=\{1\}$, and we write $M_1=\{a_i\}_{i\in I}$ for some indexing set $I$. If a graded partial monoid is generated in degree $1$, then the length $\ell(w)$ for $w\in M$ is the common length $k$ of all expressions $m=a_{i_1}\cdots a_{i_k}$. 
For such a partial monoid we write $\fac{w}$ for the set of $(i_1,\ldots,i_k)$ such that $w=a_{i_1}\cdots a_{i_k}$, and for $w\in M_k$ we write $\last{w}$ for the set of $i$ such that $w=w'a_i$ for some $w'\in M_{k-1}$. 
If such a $w'$ is always unique then we say furthermore that $M$ is \emph{right-cancelative}, and we denote this element by $w/i$. 
Finally, we say that such an $M$ has finite factorizations if we always have $|\fac{w}|<\infty$ (or equivalently if we always have $|\last{w}|<\infty$).

%%%%%%%%%%%%%%%%%%%%%%%%%
\subsection{Divided difference pairs} 
\label{subsec:dd-pairs}
%%%%%%%%%%%%%%%%%%%%%%%%%

We now formalize the relationship between the divided difference operators $\partial_i$ and the partial monoid $S_{\infty}$ in what we call a ``divided difference pair'' (\newddpair). It is not our goal to give the most general results possible, but to have a formalism that encompasses all examples we want to treat while being possibly useful in other situations.\smallskip

We fix a polynomial endomorphism $X\in \End(\poly)$ that is of degree $-1$,  i.e. $X$ takes degree $d$ homogenous polynomials to degree $d-1$ homogenous polynomials for all $d$.

For any $i\geq 1$, we define the shifted operator $X_i\in \End(\poly)$ by the composition
$$X_i:\poly\cong \poly_{i-1}\otimes \poly \to \poly_{i-1}\otimes \poly \cong \poly$$
where the first and last isomorphisms are given by the isomorphism $$\poly_{i-1}\otimes \poly=\ZZ[x_1,\ldots,x_{i-1}]\otimes \ZZ[x_i,x_{i+1},\ldots]\cong \poly,$$ and the middle map is given by $\idem\otimes X$. In particular $X=X_1$ and we always have
\begin{align}
\label{eq:vanishing}
    f\in \poly_n\implies X_{n+1}f=X_{n+2}f=\cdots=0,
\end{align}
since in this case $X$ acts on constants, and thus vanishes as it has degree $-1$.

\begin{eg}
    If we set $\partial\in \End(\poly)$ to be the first divided difference
    \begin{align}
\label{eqn:partial}\partial(f)=\frac{f(x_1,x_2,x_3,\ldots)-f(x_2,x_1,x_3,\ldots)}{x_1-x_2},\end{align}
    then $\partial_i$ agrees with \eqref{eqn:partiali}. 
\end{eg}

Note that $\partial$ is called the divided difference operator because the formula involves dividing a difference by a linear form. The way in which the various $X$ we consider in later sections arise will also be from taking two degree $0$ operators $A,B\in \End(\poly)$ such that $(A-B)f$ is always divisible by a linear form $L$, and then setting $X=\frac{A-B}{L}$.

Writing dd for divided difference, we call $X$ and the $X_i$ \textit{dd-operators} even if they do not necessarily arise in this way in general.

\begin{defn}
\label{defn:ddform}
    We define a \emph{divided difference pair} (or a dd-pair) to be the data of $(X,M)$ where $M$ is a graded right-cancelative partial monoid, generated in degree $1$ by $\{a_i\}_{i\geq 1}$, such that the map $a_i\mapsto X_i$ is a representation of $M$. 
 For $w\in M$ we write $X_w$ for the associated endomorphism of $\poly$, and in particular we have $X_i=X_{a_i}$.
\end{defn}
\begin{eg}
\label{eg:partial}
    If we set $M=S_{\infty}$ with its partial monoid structure given by \eqref{eqn:Sinftypartial}, $\partial$ as in \eqref{eqn:partial}, then the divided difference representation $s_i\mapsto \partial_i$ makes $(\partial,S_{\infty})$ into a \newddpair.
\end{eg}
\begin{eg}
    For any degree $-1$ polynomial endomorphism $X$ we have $(X,M)$ is a \newddpair for $M$ the free monoid on $\{1,2,\ldots\}$.
\end{eg}
\begin{eg}
\label{eg:derivatives}
$\nvect$ is a monoid via componentwise addition, and we have a representation given by $i\mapsto \frac{d}{dx_i}$ because $\frac{d}{dx_i}\frac{d}{dx_j}=\frac{d}{dx_j}\frac{d}{dx_i}$. Therefore $(\frac{d}{dx},\nvect)$ is a dd-pair and for $c=(c_1,\dots,c_k,0,\ldots)$ we have $\left(\frac{d}{dx}\right)_c=\left(\frac{d}{dx_1}\right)^{c_1}\cdots \left(\frac{d}{dx_k}\right)^{c_k}$.
\end{eg}

We are especially interested in the case where $M$ encodes all additive relations between compositions of the operators $X_i$. However this is a hard thing to show in general, so we do not want to assume it from the beginning. It will actually follow from the formalism we now introduce (see \Cref{thm:creation_plus_code_equal_magic}).

%%%%%%%%%%%%%%%
\subsection{Dual families of polynomials to a \newddpair}
%%%%%%%%%%%%%%%

We now generalize the relation between  $\schub{w}$ and the $\partial_i$ to an arbitrary \newddpair $(X,M)$.

%%%%%%%%%%%%%
\begin{defn}
\label{defn:dualpoly}
    A family $(S_w)_{w\in M}$ of homogenous polynomials in $\poly$ is \emph{dual} to a \newddpair $(X,M)$ if $S_{1}=1$, and for each $w\in M$ and $i\in \{1,2,\ldots\}$ we have \[
    X_iS_w=\begin{cases}S_{w/i}&\text{if }i\in \last{w}\\0&\text{otherwise.}\end{cases}
    \]
\end{defn}
%%%%%%%%%%%%%

%%%%%%%%%%%%%
\begin{eg}
    The Schubert polynomials $\{\schub{w}\suchthat w\in S_{\infty}\}$ are dual to the \newddpair $(\partial,S_{\infty})$.
\end{eg}
\begin{eg}
    If we had defined everything over $\QQ$ instead of $\ZZ$ then $\{\frac{\sfx^c}{c!}\suchthat c\in \nvect\}$ would be dual to to the \newddpair $(\frac{d}{dx},\nvect)$.
\end{eg}
%%%%%%%%%%%%%

The terminology is justified by item~\ref{it:2.4.3} of the following result.

%%%%%%%%%%%%%
\begin{prop}
\label{prop:dualpolyproperties}
    If a \newddpair $(X,M)$ has a dual family  $\{S_w\suchthat w\in M\}$, then
    \begin{enumerate}[align=parleft,left=0pt,label=(\arabic*)]
    \item \label{it:2.4.0} $M$ has finite factorizations.
    \item \label{it:2.4.1} The polynomials $S_w$ are $\ZZ$-linearly independent.
    \item \label{it:2.4.2} The representation of $\ZZ[M]$ is faithful:
    $$\sum  c_wX_w=0\implies c_w=0\text{ for all }w.$$
    In particular $M$ is the partial monoid of compositions generated by the operators $X_i$.
    \item \label{it:2.4.3} 
    Letting $\ct:\poly\to \ZZ$ be the map $f\mapsto f(0,0,\dots)$, 
    we have $\ct X_vS_{w}=\delta_{v,w}$. 
    As a consequence, for $f\in \ZZ\{S_w\suchthat w\in M\}$, the $\ZZ$-span of the $S_w$, we have
    \begin{align}
    \label{eqn:ctreconstruct}f=\sum_{w\in M}(\ct X_w f)S_w.\end{align}
\end{enumerate}
\end{prop}
\begin{proof}
First, note that for $(i_1,\ldots,i_k)\in \fac{w}$ we have $X_{i_1}\cdots X_{i_k}S_w=S_1=1$. Now we know that for any polynomial $f$ there are only finitely many $X_i$ such that $X_if\ne 0$. Applying this repeatedly we see there are only finitely many sequences $(i_1,\ldots,i_k)$ such that $X_{i_1}\cdots X_{i_k}S_w\ne 0$. Therefore $|\fac{w}|<\infty$ and \ref{it:2.4.0} is proved.

The defining relations for $S_w$ imply that $X_vS_w=S_u$ if there exists a $u\in M$ (necessarily unique by right-cancelability) such that $w=vu$, and $0$ otherwise. 
Since $S_u$ is homogenous of degree $\ell(u)$, we have $\ct S_u=\delta_{1,u}$, so $\ct X_vS_w=\delta_{v,w}$, establishing the first part of \ref{it:2.4.3}. 
This implies that the linear functionals $\{\ct X_v\suchthat v\in M\}$ are dual to the family of polynomials $\{S_w\suchthat w\in M\}$, so the polynomials $\{S_w\suchthat w\in M\}$ are linearly independent and the linear functionals $\{\ct X_w\suchthat w\in M\}$ are linearly independent, establishing \ref{it:2.4.1} and \ref{it:2.4.2}. Finally, for $f$ in the $\ZZ$-span of the $S_w$, if we write $f=\sum b_vS_v$ then applying $\ct X_w$ to both sides shows $b_w=\ct X_w f$ which implies the reconstruction formula \eqref{eqn:ctreconstruct}.
\end{proof}
%%%%%%%%%%%%%

%%%%%%%%%%%%%
\begin{eg}
    We give an example of a dd-pair whose dual family does not span $\poly$.
    Let $\partial'=\partial_2$. For the \newddpair $(\partial',S_{\infty})$ where $s_i\mapsto (\partial')_i=\partial_{i+1}$, for each $\lambda\in \ZZ$ we can construct a dual family of polynomials $S_{w}^{(\lambda)}=\schub{w}(\lambda x_1+x_2,\lambda x_1+x_3,\ldots)$.
    For no $\lambda$ does this family of polynomials span $\poly$ since $x_1$ is not in the span of the linear polynomials.
\end{eg}
%%%%%%%%%%%%%

%%%%%%%%%%%%%
\begin{eg}
    The analogue of the above theorem still holds if we had used $\QQ$ instead of $\ZZ$ in our setup. In this case, the existence of the dual family of monomials $\frac{\sfx^c}{c!}$ to the \newddpair $(\frac{d}{dx},\nvect)$ shows that the representation of $\nvect$ is faithful, and \eqref{eqn:ctreconstruct} recovers the Taylor expansion of any rational polynomial $f$: 
    $$f=\sum_c \left(\ct \left(\frac{d}{dx}\right)_c f\right)\frac{\sfx^c}{c!}.$$
\end{eg}
%%%%%%%%%%%%%

%%%%%%%%%%%%%%%%%%%%%%%%%%%%%%%%%
\subsection{Creation operators and code maps}
%%%%%%%%%%%%%%%%%%%%%%%%%%%%%%%%%

Given a \newddpair, an outstanding remaining question is whether they do admit a dual family of polynomials $S_w$. 
We give an answer in several cases of interest, using the existence of certain ``creation operators''.

%%%%%%%%%%%
\begin{defn}
    We define \emph{creation operators} for the operator $X$ to be a collection of degree $1$ polynomial endomorphisms $Y_i\in \End(\poly)$ such that on the ideal $\poly^+\subset \poly$, we have the identity
    \begin{align}\label{eq:creation_operators}
    \sum_{i=1}^{\infty} Y_iX_i=\mathrm{id}.
    \end{align}
We will also say that a \newddpair $(X,M)$ has creation operators when the operator $X$ has.
\end{defn}
%%%%%%%%%%%

Note that the left-hand side of~\eqref{eq:creation_operators} is well defined thanks to~\eqref{eq:vanishing}. 

%%%%%%%%%%%
\begin{rem}
    Note that the left-hand side of~\eqref{eq:creation_operators} vanishes on $\ZZ$, so the identity extends uniquely to $\poly$ by subtracting $\ct$ from the right-hand side, i.e. it reads $\sum_{i=1}^{\infty} Y_iX_i=\idem-\ct.$
\end{rem}
%%%%%%%%%%%

%%%%%%%%%%%
\begin{prop}\label{prop:creation_implies_stuff}
    If a \newddpair $(X,M)$ has creation operators $Y_i$ and a family of dual polynomials $\{S_w\suchthat w\in M\}$, then for $w\in M$ we have
    \begin{align}
    \label{eq:Sw_via_Yi}
    S_w=\sum_{(i_1,\dots,i_k)\in \fac{w}} Y_{i_k}\cdots Y_{i_1}(1).
    \end{align}
\end{prop}
\begin{proof}
$M$ has finite factorizations by \Cref{prop:dualpolyproperties}, so the right-hand side in~\eqref{eq:Sw_via_Yi} is well defined. 
To prove it, we induct on the length $k=\ell(w)$. 
For $k=0$  this is the identity $S_{1}=1$ and for $k>0$ we have
\begin{align*}
S_w=
\sum_{i=1}^{\infty}Y_iX_iS_w=\sum_{i\in \last{w}}Y_iS_{w/i}
&=\sum_{i\in \last{w}}\sum_{(i_1,\dots,i_{k-1})\in \fac{w/i}}Y_iY_{i_{k-1}}\cdots Y_{i_1}(1)
\nonumber\\
&=\sum_{(i_1,\dots,i_k)\in \fac{w}}Y_{i_k}\cdots Y_{i_1}(1).\qedhere
\end{align*}
\end{proof}
%%%%%%%%%%%

An immediate consequence is that if a \newddpair has creation operators, it has at most one dual family of polynomials. The creation operators are not unique in general, and this leads to possibly distinct expansions of $S_w$ as we will see in later sections.

%%%%%%%%%%%
\begin{eg}
\label{eg:ddxcreator}
If we had used $\QQ$ instead of $\ZZ$ in our setup, then for $(\frac{d}{dx},\nvect)$, we can take $Y_i$ to act on homogenous polynomials of degree $k$ by $Y_i(f)=\frac{1}{k+1}x_if$ for all $k$. 
Then \eqref{eq:creation_operators} holds as it is Euler's famous theorem $\sum x_i\frac{d}{dx_i}=k\idem$ on homogenous polynomials of positive degree $k$. 
For $c=(c_i)_{i\geq 1}\in \nvect$, we have $\fac{c}=\{(i_1,\ldots,i_k)\suchthat 
 c_p=\#\{1\leq j\leq k\suchthat i_j=p\}\}$, and~\eqref{eq:Sw_via_Yi} recovers the formula $S_c=\frac{\sfx^c}{c!}$ for the unique candidate family of polynomials satisfying \eqref{eqn:ddxdualrecursion}.
\end{eg}
%%%%%%%%%%%

 Let us give an example now to show that the existence of creation operators is not enough to ensure the existence of a dual family of polynomials.

%%%%%%%%%%%
\begin{eg}
 Define $X$ by linearly extending the assignments $X(x_i)=\delta_{i,1}$ for all $i\ge 1$, and some degree $-1$ injection $\Phi$ on monomials of degree $d$ to monomials of degree $d-1$ for each $d\ge 2$. We can assume that $x_1$ does not occur in the range of $\Phi$, by applying the shift $x_i\mapsto x_{i+1}$ if necessary. $X$ has the following creation operators $Y_i$: on the constant polynomials, $Y_i$ is multiplication by $x_i$. On $\poly^+$, define $Y_2=Y_3=\cdots=0$ while $Y_1$ equals $\Phi^{-1}$ on monomials in the range of $\Phi$, and $0$ on the remaining monomials.
 
    If $(S_w)_{w\in M}$ is dual to some \newddpair $(X,M)$, we have $S_{a_1}=x_1$. Now $a_1\cdot a_1$ is defined in $M$ since $X_1^2=X^2$ is nonzero, and we have $X_1(S_{a_1\cdot a_1})=x_1$. This is not possible by our assumption on $\Phi$, and thus $(X,M)$ does not have a dual family.
\end{eg}
%%%%%%%%%%%

We now give a simple to check hypothesis on $M$ to ensure that the dual polynomials do in fact exist and, furthermore, form a basis of $\poly$.
   
Let $\nvect$ denote the set of finitely supported sequences of nonnegative integers $c=(c_1,c_2,\ldots)$. 
For $c\in \nvect$, write $\supp c$ for the set of $i$ such that $c_i\neq 0$, and $|c|$ for the sum of the nonzero entries. Let $M$ be a graded right-cancelable monoid.

%%%%%%%%%%%
\begin{defn}
\label{def:codemap}
    A \emph{code map} for $M$ is an injective map $c:M\to \nvect$  such that $\ell(w)=|c(w)|$ and $\max \supp c(w)=\max \last{w}$ for all $w\in M$. (In particular $M$ has finite factorizations.)
\end{defn}
%%%%%%%%%%%
We note that the existence of a code map is trivially seen to be equivalent to the condition that
$$\#\{w\in M\suchthat \ell(w)=n\text{ and }\max \last{w}=k\}\le \#\{c\in \nvect\suchthat |c(w)|=n\text{ and }\max\supp c(w)=k\}$$
but in practice verifying code maps exist seems to be more straightforward than checking this inequality by other means.

%%%%%%%%%%%
\begin{thm}
\label{thm:creation_plus_code_equal_magic}
    Suppose that a \newddpair $(X,M)$ has creation operators and a code map. Then 
    \begin{enumerate}[align=parleft,left=0pt,label=(\arabic*)]
        \item \label{they_exist:it1} The code map is bijective.
        \item \label{they_exist:it2} There is a unique dual family $(S_w)_{w\in M}$ defined by \eqref{eq:Sw_via_Yi}. It is a basis of $\poly$.
        \item \label{they_exist:it3} The subfamily $(S_w)_w$ where $\max \supp c(w)\le d$ is a basis of $\poly_d$ for any $d\geq 0$.
    \end{enumerate}
\end{thm}
\begin{proof}
Define recursively $S_1=1$ and 
\[
S_w=\sum_{i\in \last{w}}Y_iS_{w/i}.
\]
By \Cref{prop:creation_implies_stuff}, the dual family of polynomials must be equal to $\{S_w\suchthat w\in M\}$ if it exists.

We begin by addressing~\ref{they_exist:it1}.
Let $$M_{k,d}=\{w\in M\suchthat \ell(w)=k\text{ and }\max\supp c(w)\le d\}.$$
We claim that for $f\in \poly_d^{(k)}$, the homogenous degree $k$ polynomials in $\poly_d$, we have
\begin{align}
\label{eqn:Mkdequality}
    f=\sum_{w\in M_{k,d}}X_w(f)S_w.
\end{align}
By induction on $k$ we can show \eqref{eqn:Mkdequality} but with $w\in M_{k,d}$ replaced with the condition $\ell(w)=k$ since
$$f=\sum_{i=1}^{\infty} Y_iX_if=\sum_{i=1}^{\infty}Y_i\sum_{\ell(w')=k-1}(X_{w'}X_if)S_{w'}=\sum_{\ell(w)=k}\sum_{i\in \last{w}}Y_i(X_w(f)S_{w/i})=\sum_{\ell(w)=k}X_w(f)S_w.$$
To conclude, it suffices to show that if $\ell(w)=k$ and $w\not\in M_{k,d}$ then $X_wf=0$; this is true because if $i=\max \supp c(w)>d$, then $i\in \last{w}$ and so $X_wf=X_{w/i}X_if=0.$

Writing $\nvect_{k,d}=\{c\in \nvect \suchthat \max\supp c\le d\text{ and }|c(w)|=k\}$, the code map induces an injection $M_{k,d}\to \nvect_{k,d}$ so $|M_{k,d}|\le |\nvect_{k,d}|$. 
On the other hand,~\eqref{eqn:Mkdequality} implies the inclusion \begin{align}\label{eqn:polydkcontainment}\poly_d^{(k)}\subset \ZZ\{S_w\suchthat w\in M_{k,d}\},\end{align} so $|\nvect_{k,d}|=\operatorname{rank} \poly_d^{(k)}\le |M_{k,d}|$. 
We conclude that $|M_{k,d}|=|\nvect_{k,d}|=\operatorname{rank} \poly_d^{(k)}$, implying~\ref{they_exist:it1} and the fact the $S_w$ are $\ZZ$-linearly independent.

Observe that~\eqref{eqn:polydkcontainment} is a containment of equal rank free abelian groups.
Furthermore $\poly_d^{(k)}$ is saturated (i.e. for any $\lambda\in \ZZ$ we have $\lambda f\in \poly_d^{(k)}$ implies $f\in \poly_d^{(k)}$), so the containment \eqref{eqn:polydkcontainment} is in fact an equality and we conclude that $\{S_w\suchthat w\in M_{k,d}\}$ is a $\ZZ$-basis of $\poly_d^{(k)}$. 
Taking the union of these bases for all $k$ and fixed $d$ shows that $\{S_w\suchthat \max \supp c(w)\le d\}$ is a basis for $\poly_d$, which shows ~\ref{they_exist:it3}.

By considering these basis statements and the identity \eqref{eqn:Mkdequality} for growing $d$, and using the fact that $\bigcup M_{k,d}=M$, we deduce that $\{S_w\suchthat w\in M\}$ is a basis of $\poly$, proving the second half of~\ref{they_exist:it2}. For arbitrary $f\in \poly$ we have the identity
$$f=\sum_{w\in M}(\ct X_wf)S_w.$$
We thus infer that
\begin{enumerate}[label=(\alph*)]
    \item \label{first point} If $\ct X_wf=0$ for all $w\in M$ then $f=0$, and
    \item \label{second point} $S_w$ is the unique polynomial such that $\ct X_{w'}S_w=\delta_{w',w}$ for all $w'\in M$.
\end{enumerate}
We are now ready to show that $X_iS_w=\delta_{i\in \last{w}}S_{w/i}$ for any $i$ and $w$. If $w'\in M$, we have
$$\ct X_{w'}(X_iS_w)=\ct X_{w'\cdot i}S_w=\delta_{w,w'\cdot i}.$$
Here the last two terms are considered as zero if $w'\cdot i$ is not defined. If $i\not \in \last{w}$ then $\delta_{w,w'\cdot i}=0$ for all $w'\in M$ so we conclude by~\ref{first point} that $X_iS_w=0$. 
On the other hand, if $i\in \last{w}$ then $\delta_{w,w'\cdot i}=\delta_{w/i,w'}$ which by~\ref{second point} implies $X_iS_w=S_{w/i}$ as desired.
\end{proof}
%%%%%%%%%%%

%%%%%%%%%%%
\begin{eg}
 For $(\frac{d}{dx},\nvect)$ there is a code map on $\nvect$ given by the identity. Therefore using $\QQ$ instead of $\ZZ$ in our setup, we can conclude that $\{S_c=\frac{\sfx^c}{c!}\suchthat c\in \nvect\}$ found in \Cref{eg:ddxcreator} is the dual family of polynomials to $(\frac{d}{dx},\nvect)$ without directly verifying the recursion \eqref{eqn:ddxdualrecursion}.
\end{eg}
%%%%%%%%%%%

%%%%%%%%%%%%%%%%%%%%%%%%%%%%%%%%%%%%%%%%%%%%%%%%%%%
\section{Schubert polynomials}
\label{sec:schubert}
%%%%%%%%%%%%%%%%%%%%%%%%%%%%%%%%%%%%%%%%%%%%%%%%%%%

The divided difference $\partial_i\in \End(\poly)$ for $i=1,2,\dots$ is defined as follows:
\begin{align*}
    \partial_if(x_1,x_2,\ldots)&=\frac{f-f(x_1,\dots,x_{i-1},x_{i+1},x_{i},\ldots)}{x_i-x_{i+1}}.
\end{align*}
The partial monoid $M$ is given by the nil-Coxeter monoid $S_{\infty}$ of permutations of $\{1,2,\ldots\}$ fixing all but finitely many elements with partial product $u\circ v=uv$ if $\ell(u)+\ell(v)=\ell(uv)$, undefined otherwise: here $\ell$ and $uv$ are the lengths and product in the group $S_{\infty}$. 
Denoting the simple transposition $s_i=(i,i+1)$, the corresponding \newddpair $(\partial,S_{\infty})$ comes from the representation $s_i\mapsto \partial_i$.

We have $$\last{w}=\des{w}=\{i\suchthat w(i)>w(i+1)\},$$
and $\fac{w}=\red{w}$, the set of reduced words for $w$, i.e. the set of sequences $(i_1,\ldots,i_k)$ with $k=\ell(w)$ such that $w=s_{i_1}\cdots s_{i_k}$. 
The Lehmer code is the bijective map $S_{\infty}\to \nvect$ defined for $w\in S_{\infty}$ by $\lcode{w}=(c_1,c_2,\ldots)$ where $c_i=\#\{j>i\suchthat w(i)>w(j)\}$. 
Because $\des{w}=\{i\suchthat c_i>c_{i+1}\}$, we have $\max\supp \lcode{w}=\max\last{w}$, so this is a code map as in \Cref{def:codemap}.

The Schubert polynomials are the unique family of homogenous polynomials dual to the \newddpair $(\partial, S_{\infty})$: we have $\schub{\idem}=1$ and
$$\partial_i\schub{w}=\begin{cases}\schub{w/i}& \text{if }i\in \des{w}\\0&\text{otherwise.}\end{cases}$$
Figure~\ref{fig:applying_dels} shows the application of various divided difference operators starting from $\schub{1432}$.
\begin{figure}[!ht]
    \centering
    \includegraphics[width=\textwidth]{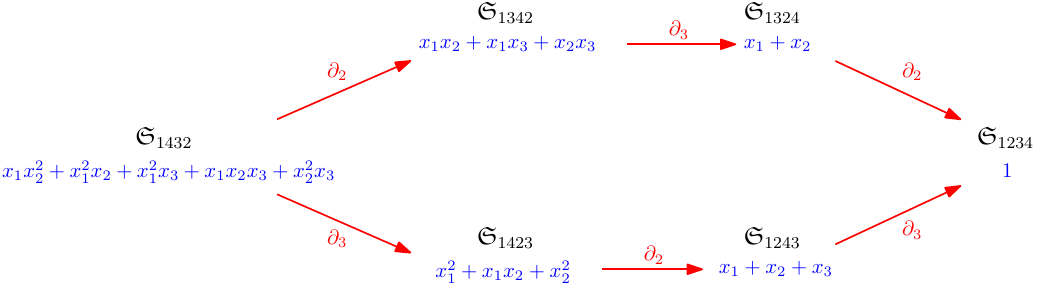}
    \caption{Sequences of $\partial_i$ applied to a  $\schub{w}$}
    \label{fig:applying_dels}
\end{figure}

The standard way the existence of Schubert polynomials is shown is through the Ansatz
$\schub{w_{0,n}}=x_1^{n-1}x_2^{n-2}\cdots x_{n-1}$
for $w_{0,n}$ the longest permutation in $S_n$. Because every $u\in S_{\infty}$ has $u\le w_{0,n}$ for some $n$, it turns out it suffices to check that $\partial_{w_{0,n-1}^{-1}w_{0,n}}x_1^{n-1}x_2^{n-2}\cdots x_{n-1}=x_1^{n-2}x_2^{n-3}\cdots x_{n-2}$, which is done with direct calculation. 

Using our setup, because there is a code map we can simultaneously avoid the Ansatz and establish an explicit combinatorial formula by exhibiting creation operators for the $\partial_i$.

%%%%%%%%%%%%%%%%%%%%%%%%%%%%%%%%%%%%%%%%%%%%%%%%%%
\subsection{Creation operators for $\partial_i$}
\label{subsec:partialcreation}
%%%%%%%%%%%%%%%%%%%%%%%%%%%%%%%%%%%%%%%%%%%%%%%%%%%

We now describe creation operators for $\partial_i$, which will give formulas for the Schubert polynomials. We define the \emph{Bergeron--Sottile map} \cite{BS98}
$$\rope{i}f(x_1,x_2,\ldots)=f(x_1,\ldots,x_{i-1},0,x_i,\ldots).$$
%%%%%%%%%%
\begin{lem}\label{le:xr_equals_id}
We have
    \begin{align*}
        \sum_{i\geq 1} x_i\rope{i}\partial_i=
        \idem-\rope{1}.
    \end{align*}
\end{lem}
\begin{proof}
    We sum the relation $x_i\rope{i}\partial_i=\rope{i+1}-\rope{i}$ for all $i\geq 1$.
\end{proof}
%%%%%%%%%%

We define $$\zope=\idem+\rope{1}+\rope{1}^2+\cdots:\poly^+\to \poly^+.$$

%%%%%%%%%%
\begin{cor}
\label{cor:ZxRcreation}
We have that $\zope x_i\rope{i}$ are creation operators for the \newddpair given by the usual divided differences $\partial_i$ and the nil-Coxeter monoid. That is, the identity
    $$\sum_{i\ge 1}\zope x_i\rope{i}\partial_i=\idem$$
    holds on $\poly^+$.
    In particular, Schubert polynomials exist and we have the following monomial positive expansion
    $$\schub{w}=\sum_{(i_1,\ldots,i_k)\in \red{w}}\zope x_{i_k}\rope{i_k}\cdots \zope x_{i_1}\rope{i_1}(1).$$
\end{cor}
\begin{proof} We compute $\zope\sum_{i\geq 1} x_i\rope{i}\partial_i=Z(\idem-\rope{1})=(\idem-\rope{1})+\rope{1}(\idem-\rope{1})+\cdots=\idem$.
\end{proof}
%%%%%%%%%%

%%%%%%%%%%
\begin{eg}\label{eg:schubert}
    Take $w=14253$ so that $\red{w}=\{324,342\}$. 
    Adopting the shorthand $\zope \textsf{x}\rope{\mathbf{i}}$ for composite $\zope x_{i_k}\rope{i_k}\cdots \zope x_{i_1}\rope{i_1}$ where $\mathbf{i}=(i_1,\dots,i_k)$ one gets
    \begin{align*}
    &\zope\textsf{x}\rope{(3,2,4)}(1)=
    \zope\textsf{x}\rope{(2,4)}(x_1+x_2+x_3)=
    \zope\textsf{x}\rope{(4)}(x_1x_2+x_1^2+x_2^2)=x_1x_2x_4+x_1^2x_4+x_1^2x_3+x_2^2x_4
    \\
   &\zope\textsf{x}\rope{(3,4,2)}(1)=
    \zope\textsf{x}\rope{(4,2)}(x_1+x_2+x_3)=
    \zope\textsf{x}\rope{(2)}(x_1x_2+x_1x_3+x_1x_4+x_2x_3+x_2x_4+x_3x_4)\\&\qquad\qquad\qquad\qquad\qquad\qquad\qquad\qquad=x_1x_2^2+x_1x_2x_3+x_2^2x_3+x_1^2x_2.
    \end{align*}
    On adding the two right-hand sides one obtains the Schubert polynomial $\schub{14253}$.
\end{eg}
%%%%%%%%%%

%%%%%%%%%%
\begin{rem}
    The slide expansion of Schubert polynomials \cite{AssSea17,BJS93}, reproved in \Cref{prop:slide_exp}, expresses $\schub{w}$ as a sum of slide polynomials over $\red{w}$. Corollary~\ref{cor:ZxRcreation} also provides an expression where the sum ranges over $\red{w}$, but these two decompositions are in fact distinct, as the preceding example reveals as neither $\zope\textsf{x}\rope{(3,2,4)}(1)$ nor $\zope\textsf{x}\rope{(3,4,2)}(1)$ equals a slide polynomial.
\end{rem}
%%%%%%%%%%

%%%%%%%%%%%%%%%%%%%%%%%%%%%%%%%%%%%%%%%%%%%%%%%%%%
\subsection{Pipe dream interpretation}
\label{subsec:pipedream}
%%%%%%%%%%%%%%%%%%%%%%%%%%%%%%%%%%%%%%%%%%%%%%%%%%

We now relate the preceding results to a simple bijection at the level of pipe dreams. 
Consider the staircase $\textsf{Stair}_n\coloneqq (n,n-1,\dots,1)$ whose columns are labeled $1$ through $n$ left to right.
Given $w\in S_n$, a (reduced) pipe dream for $w$ is a tiling of $\textsf{Stair}_n$ using `cross' and `elbow' tiles depicted in Figure~\ref{fig:pd_tiles_example} so that the following conditions hold:
\begin{itemize}
    \item The tilings form $n$ pipes with the pipe entering in row $i$ exiting via column $w(i)$ for all $1\leq i\leq n$;
    \item No two pipes intersect more than once.
\end{itemize}

\begin{figure}[!ht]
    \centering
    \includegraphics{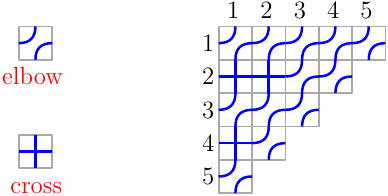}
    \caption{Elbow and cross tiles (left) and a pipe dream for $w=14253$ (right)} 
    \label{fig:pd_tiles_example}
\end{figure}

Denote the set of pipe dreams for $w$ by $\pd{w}$. Given $D\in \pd{w}$ attach the monomial 
\[
    \sfx^{D}\coloneqq \prod_{\text{crosses} (i,j)\in D} x_i.
\]
A famous result of Billey--Jockusch--Stanley \cite{BJS93} (see also \cite{BerBil93, FK96,FS94}) then states that
%%%%%%%%%%
\begin{thm}
\label{thm:bjs}
    $\schub{w}$ is the generating polynomial for pipe dreams for $w$:
    \[
    \schub{w}=\sum_{D\in \pd{w}}\textsf{x}^{D}.
\]
\end{thm}
%%%%%%%%%%

We will give a simple proof, using the recursion
\begin{align}
\label{eq:what_phi_sees}
\schub{w}=\rope{1}\schub{w}+\sum_{i\in \des{w}}x_i\rope{i}\schub{ws_i}.
\end{align}
which follows immediately from \Cref{le:xr_equals_id} and the definition of Schubert polynomials.

\begin{proof}[Proof of \Cref{thm:bjs}]
We need to show
    $$\sum_{D\in \pd{w}}\sfx^D=\rope{1}\sum_{D\in \pd{w}}\sfx^D+\sum_{i\in \des{w}}x_i\rope{i}\sum_{D\in \pd{ws_i}}\sfx^D.$$
Say that a pipe dream $D\in \pd{w}$ is uncritical if there are no crosses in column $1$, and $i$-critical if the last cross in column $1$ is in row $i$. Denote $\pd{w}^0\subset \pd{w}$ for the set of uncritical pipe dreams, and $\pd{w}^i\subset \pd{w}$ for the set of $i$-critical pipe dreams.

Note that if $i\ge 1$ and $\pd{w}^i$ is nonempty, then $i\in \des{w}$ since pipes $i$ and $i+1$ cross at the location of this last cross in column $1$. 
Because $\pd{w}=\bigsqcup \pd{w}^i$, it suffices to show that
\begin{enumerate}[label=(\alph*)]
    \item \label{first of these points} $\sum_{D\in \pd{w}^0 }\sfx^D=\rope{1}\sum_{D\in \pd{w}}\sfx^D$ and
    \item \label{second of these points} for $i\in \des{w}$ we have $\sum_{D\in \pd{w}^i}\sfx^D=x_i\rope{i}\sum_{D\in \pd{ws_i}}\sfx^D$.
\end{enumerate}
To see~\ref{first of these points} we note there is a weight-preserving bijection $$\Phi_0:\pd{w}^0\to \{D\in \pd{w}\suchthat D\text{ has no crosses in row $1$}\},$$ given by shifting all crosses one unit diagonally southwest. 
Since $\sfx^D=\rope{1}x^{\Phi_0(D)}$, we have~\ref{first of these points}.

To see~\ref{second of these points}, we note there is a bijection $$\Phi_i:\pd{w}^i\to \{D\in \pd{ws_i}\suchthat \text{D has no crosses in row $i$}\}$$
obtained by turning the last cross in column $1$ into an elbow and then shifting all crosses in rows $i$ and below one unit diagonally southwest. 
See Figure~\ref{fig:the_Phi_map} for an illustration.
As $\sfx^{D}=x_i\rope{i}\sfx^{\Phi_i(D)}$, we have~\ref{second of these points}.
\end{proof}
%%%%%%%%%%%

\begin{figure}[!ht]
    \centering
    \includegraphics{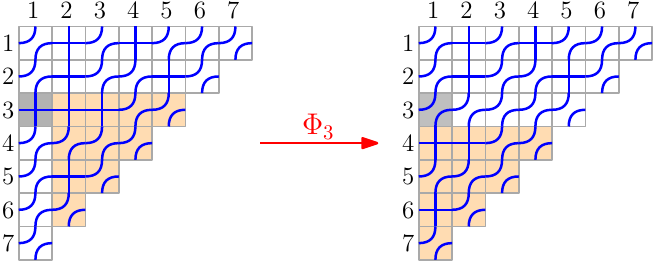}
    \caption{A $3$-critical pipe dream $D$ for $w=1375264$ (left), and $\Phi_3(D)\in \operatorname{PD}(ws_3)$ }
    \label{fig:the_Phi_map}
\end{figure}

\begin{rem}
    Since the image of $\partial_i$ comprises polynomials symmetric in $\{x_i,x_{i+1}\}$ we can replace the $\rope{i}\partial_i$ in Lemma~\ref{le:xr_equals_id} by $\rope{i+1}\partial_i$. The recursion in~\eqref{eq:what_phi_sees} is then equivalent to
    \begin{align}\label{eq:R_i_plus_one}
        \schub{w}=\rope{1}\schub{w}+\sum_{i\in \des{w}}x_i\rope{i+1}\schub{ws_i}.         
    \end{align}
    In private communication with the authors, Dave Anderson has sketched a representation-theoretic proof of the recursion in~\eqref{eq:R_i_plus_one} using Kra\'skiewicz--Pragacz modules \cite{KraPra1987foncteurs, KraPra2004Schubert}.
\end{rem}

%%%%%%%%%%%%%%%%%%%%%%%%%%%%%%%%%%%%%%%%%%%%%%%%%%%%%%%%%%
\section{Forest polynomials}
\label{sec:Forest}
%%%%%%%%%%%%%%%%%%%%%%%%%%%%%%%%%%%%%%%%%%%%%%%%%%%%%%%%%%

The \emph{quasisymmetric divided difference} \cite{NST_1} is defined as
\begin{align*}
    \tope{i}=\rope{i}\partial_i=\rope{i+1}\partial_{i}=\frac{\rope{i+1}-\rope{i}}{x_i}.
\end{align*}
The associated \newddpair $(\tope{},\indexedforests)$ from \cite{NST_1} comes from the monoid structure on the set $\indexedforests$ of plane indexed binary forests as we shall briefly recall.

A \emph{rooted plane binary tree} $T$ is a rooted tree with the property that every node has either no child, in which case we call it a \emph{leaf,} or two children, distinguished as the ``left'' and ``right'' child, in which case we call it an \emph{internal node}. 
We let $\internal{T}$ denote the set of internal nodes and let $|T|\coloneqq |\internal{T}|$ be the \emph{size} of $T$. 
The unique tree of size $0$, whose root node is also its leaf node, is denoted by $*$. 
We shall call this the \emph{trivial} tree. 

An \emph{indexed forest} $F$ is a sequence  $(T_i)_{i\geq 1}$ of rooted plane binary trees where all but finitely many $T_i$ are trivial.
If all $T_i$ are trivial, then we call $F$ the \emph{empty forest} $\varnothing$.
By identifying the leaves with $\ZZ_{\ge 1}$, going through them from left to right, one can depict an indexed forest  as shown in 
Figure~\ref{fig:binary_indexed_forest}.
We denote the set of indexed forests by $\indexedforests$.
Given $F\in \indexedforests$, we let $\internal{F}\coloneqq \sqcup_{i} \internal{T_i}$ denote its set of internal nodes. 

\begin{figure}[!ht]
    \centering
    \includegraphics[width=\textwidth]{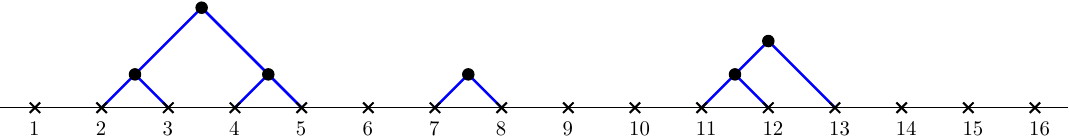}
    \caption{An indexed forest $F$ with $\sfc(F)=(0,2,0,1,0,0,1,0,0,0,2,0,0,\dots)$}
    \label{fig:binary_indexed_forest}
\end{figure}

There is a natural monoid structure on $\indexedforests$ obtained by taking $F\cdot G$ to be the indexed forest where the $i$'th leaf of $F$ is identified with the $i$'th root of $G$ for all $i$. This monoid is generated by the smallest nontrivial forests $\underline{i}$ of size $1$ with internal node having left leaf at $i$, and there is an identification of $\indexedforests$ with the (right-cancelable) Thompson monoid $\Th$ given by
$$\indexedforests \cong \Th=\langle 1,2,\cdots \suchthat i\cdot j=j\cdot (i+1)\text{ for }i>j\rangle,$$
by identifying $\underline{i}\mapsto i$.

We may encode $F\in \indexedforests$ as elements of $\nvect$ as follows.
Define $\rho_F:\internal{F}\to \ZZ_{\ge 1}$ by setting $\rho_F(v)$ equal to the label of the leaf obtained by going down left edges from $v$.
Then the map $\sfc:\indexedforests \to \nvect$ sending $F\mapsto \sfc(F)=(c_i)_{i\ge 1}$ where $c_i=\{v\suchthat \rho_F(v)=i\}$ is a bijection \cite[Theorem 3.6]{NST_1}.
The set $\last{w}$ is identified with the \emph{left terminal set} of $F$ as
\[
    \qdes{F}=\{i\suchthat c_i\neq 0\text{ and }c_{i+1}=0\},
\]
which in particular immediately implies that $\max \supp c(F)=\max \last{F}$ so $\sfc$ is a code map.
We explain the choice of name. 
We call $v\in \internal{F}$  \emph{terminal} if both its children are leaves, necessarily $i$ and $i+1$ where $i\coloneqq \rho_F(v)$. 
We then have $c_i\neq 0$ and $c_{i+1}=0$, i.e. $i\in \qdes{F}$.
Thus we can record terminal nodes by recording the label of their left leaf, which is what $\qdes{F}$ does.

For $F\in\indexedforests$ and $i\in \qdes{F}$, we call $F/i \in \indexedforests$ the \emph{trimmed forest}  (as in \cite[\S 3.6]{NST_1}), which is obtained by deleting the terminal node $v$ satisfying $\rho_F(v)=i$.
The set of factorizations $\fac{F}$ is then identified with the set of \emph{trimming sequences} \cite[Definition 3.8]{NST_1}:
\[
    \Trim{F}=\{(i_1,\dots,i_k)\suchthat (((F/i_k)/i_{k-1})/\cdots)/i_1 =\varnothing \}.
\]
Figure~\ref{fig:trimmingexample} (ignoring the polynomials in blue) shows repeated trimming operator applied to the indexed forest $F$ on the left. It follows that $\Trim{F}=\{(1,1,3),(1,2,1)\}$.

\begin{figure}[!ht]
    \centering    
    \includegraphics[width=\textwidth]{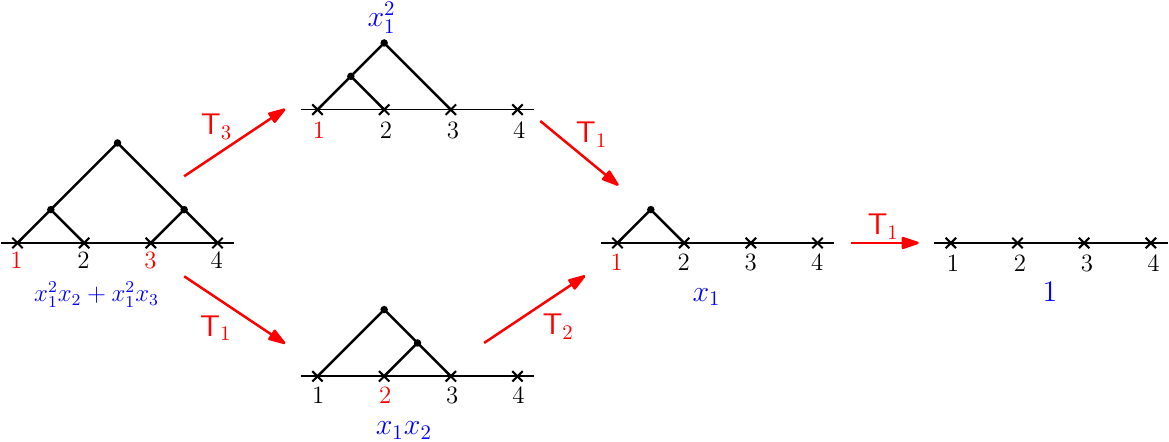}
    \caption{$\tope{i}$ applied to various $\forestpoly{F}$, with elements of $\qdes{F}$ highlighted in red}
    \label{fig:trimmingexample}
\end{figure}
\subsection{Creation operators for $\tope{i}$}
We now describe creation operators for $\tope{i}$.
\begin{thm}\label{th:creation_characterization}
    We have $\sum_{i\ge 1}\zope x_i\tope{i}=\idem$ on $\poly^+$, or in other words $\zope x_i$ are creation operators for $\tope{i}$. In particular, there is a family of ``forest polynomials'' $\forestpoly{F}$ characterized by $\forestpoly{\emptyset}=1$ and
    $$\textsf{T}_i\forestpoly{F}=\begin{cases}\forestpoly{F/i}&i\in \qdes{F}\\0&\text{otherwise,}\end{cases}$$
    with the following monomial-positive expansion
    $$\forestpoly{F}=\sum_{(i_1,\ldots,i_k)\in \Trim{F}}\zope x_{i_k}\cdots \zope x_{i_1} (1).$$
\end{thm}
\begin{proof}
    \Cref{cor:ZxRcreation} already contains this identity in the form $\sum_{i\geq 1} \zope x_i\rope{i}\partial_i=\idem$ on $\poly^+$. The rest follows from \Cref{thm:creation_plus_code_equal_magic}.
\end{proof}

Figure~\ref{fig:trimmingexample} shows the result of applying $\tope{1}\tope{1}\tope{3}$ and $\tope{1}\tope{2}\tope{1}$ to the forest polynomial $\forestpoly{F}=x_1^2x_2+x_1^2x_3$. As per Theorem~\ref{th:creation_characterization}, each application of a $\tope{}$ trims the indexed forest at that stage.

\begin{eg}\label{eg:forest}
    We shall consider the indexed forest $F$ whose corresponding forest  polynomial $\forestpoly{F}$ is computed in \cite[Example 3.9]{NT_forest}.
    This happens to be equal to $\schub{14253}$ from Example~\ref{eg:schubert}, but as we shall see the decompositions are different.
    We have $\Trim{F}=\{(2,2,4),(2,3,2)\}$.
    Adopting the shorthand $\zope \textsf{x}_{\mathbf{i}}$ for the composite $\zope x_{i_k}\cdots \zope x_{i_1}$ where $\mathbf{i}=(i_1,\dots,i_k)$ one gets
    \begin{align*}
        &\zope\textsf{x}_{(2,2,4)}(1)=
    \zope\textsf{x}_{(2,4)}(x_1+x_2)=
    \zope\textsf{x}_{(4)}(x_1x_2+x_1^2+x_2^2)=x_1x_2x_4+x_1^2x_4+x_1^2x_3+x_2^2x_4
    \\
   &\zope\textsf{x}_{(2,3,2)}(1)=
    \zope\textsf{x}_{(3,2)}(x_1+x_2)=
    \zope\textsf{x}_{(2)}(x_1x_2+x_1x_3+x_2x_3)=x_1x_2^2+x_1x_2x_3+x_2^2x_3+x_1^2x_2.
    \end{align*}
    Thus we find that $\forestpoly{F}$ is the sum of the two right-hand sides. 
    Observe that even though two final expressions above align with those computed in Example~\ref{eg:schubert}, the expressions obtained at the intermediate stages are not the same.
\end{eg}
%%%%%%%%%%

%%%%%%%%%%%%%%%%%%%%%%%%%%%%%%%%%%%%%%%%%%%%%%%%%%
\subsection{Diagrammatic Interpretation}
\label{subsec:diagram}
%%%%%%%%%%%%%%%%%%%%%%%%%%%%%%%%%%%%%%%%%%%%%%%%%%

We now give a diagrammatic perspective on forest polynomials that evokes the pipe dream perspective on Schubert polynomials. By applying the relation $\rope{1}+\sum_{i\geq 1} x_i\tope{i}=\idem$ from \Cref{cor:ZxRcreation} to forest polynomials we obtain
\begin{align}\label{one_step_recursion}
    \rope{1}\forestpoly{F}+\sum_{i\in \qdes{F}}x_i\forestpoly{F/i}=\forestpoly{F}.
\end{align}
This identity was previously obtained in \cite[Lemma 3.12]{NT_forest}.
Unwinding this recursion leads to the following combinatorial model similar to the pipe dream expansion of Schubert polynomials, which can be matched up without much difficulty to the combinatorial definitions of forest polynomials in \cite{NST_1,NT_forest}.

We will represent each of the operators $\rope{1}$ and $\tope{1},\tope{2},\ldots$ as a certain graph on a $(\ZZ_{\ge 1} \times 2)$-rectangle as shown in Figure~\ref{fig:labelings_in_two_ways} on the left.
Consider the grid $\ZZ_{\ge 1}\times \ZZ_{\ge 1}$ where we adopt matrix notation, i.e. the elements in the grid are $(i,j)\in \ZZ_{\ge 1}\times \ZZ_{\ge 1}$ where we first coordinate increases top to bottom and the second coordinate increases left to right.

We define a \emph{forest diagram} to be any graph on vertex set $\ZZ_{\ge 1}\times \ZZ_{\ge 1}$ such that the subgraph induced on the vertex set $\{(p,q)\suchthat p\in \ZZ_{\ge 1}, q\in \{k,k+1\}\}$ either represents  $\tope{i}$ for some positive integer $i$ or represents $\rope{1}$, and such that for $p$ large enough all such induced subgraphs represent $\rope{1}$.
In particular we may without loss of information restrict our attention to the finite subgraph on the vertex set $\{(i,j)\suchthat i+j\leq n+1\}$ for some $n$. 
See on the right in Figure~\ref{fig:labelings_in_two_ways} for an example.
Given any such diagram $D$ we let $\textrm{nodes}(D)$ denote the set of $(i,j)$ where we have $(i,j)$ directly connected to both $(i,j-1)$ and $(i+1,j-1)$, and associate a monomial 
\[
\textsf{x}^D\coloneqq\prod_{(i,j)\in \textrm{nodes}(D)}x_i
\]
Note that any such graph is necessarily acyclic and naturally corresponds to an indexed forest, as shown in Figure~\ref{fig:labelings_in_two_ways}.
For $F\in \indexedforests$ let $\diagrams{F}$ denote the set of diagrams whose underlying forest is $F$.

\begin{figure}
    \centering
    \includegraphics[width=\textwidth]{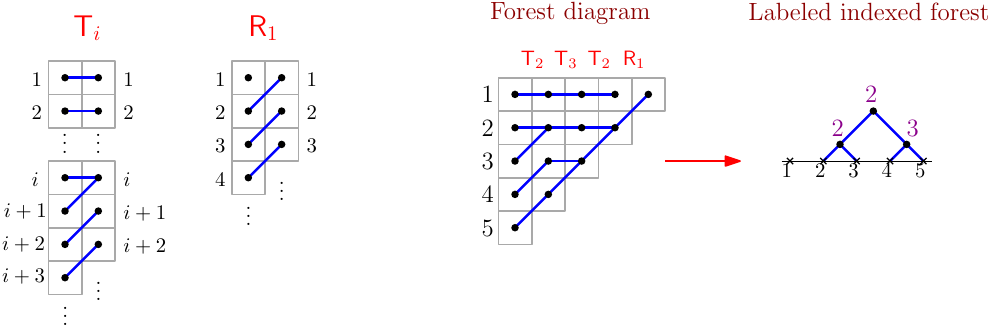}
    \caption{The graphs corresponding to $\tope{i}$ and $\rope{1}$ (left), and a forest diagram with the corresponding labeled indexed forest  (right)}
    \label{fig:labelings_in_two_ways}
\end{figure}
\begin{thm}
    For $F\in \indexedforests$ we have the forest diagram formula
    $$\forestpoly{F}=\sum_{D\in \diagrams{F}}\sfx^D.$$
\end{thm}
\begin{proof}
    We give a brisk proof sketch that the claimed expansion satisfies~\eqref{one_step_recursion} along the lines of the proof of Theorem~\ref{thm:bjs}.
    
    Call $D\in \diagrams{F}$ \emph{$i$-critical} if the subgraph induced on $\{(j,1),(j,2)\suchthat j\ge 1\}$ represents  $\tope{i}$ for some positive integer $i$. Otherwise we call $D$ \emph{uncritical}, in which case the aforementioned subgraph necessarily represents $\rope{1}$. Note that if $D$ is $i$-critical, then $i\in \qdes{F}$.
    
    Denote by $\diagrams{F}^0$ the set of uncritical forest diagrams, and by $\diagrams{F}^i$ the set of $i$-critical forest diagrams.
    Consider the weight-preserving bijection
    \[
    \Phi_0:\diagrams{F}^0\to \{D\in \diagrams{F}\suchthat \text{no element of \textrm{nodes}(D) is in row $1$}\}
    \]
    given by shifting all nodes one unit diagonally southwest. Clearly $\sfx^D=\rope{1}\sfx^{\Phi_0(D)}$.

   Consider next the bijection
   \[
    \Phi_i:\diagrams{F}^i\to \{D\in \diagrams{F/i}\}
   \]
   given by taking the subgraph induced on vertices $(p,q)$ with $p\geq 1, q\geq 2$. That is, we ignore vertices of the form $(p,1)$ as well as all incident edges. It is easily seen that $\sfx^{D}=x_i\, \sfx^{\Phi_i(D)}$.
\end{proof}
%%%%%%%%%%%%%%%%%%%%%
\subsection{$m$-forest polynomials}
\label{subsec:m_forests}
%%%%%%%%%%%%%%%%%%%%%

We now briefly touch upon the more general family of $m$-forest polynomials defined combinatorially in~\cite{NST_1}, where the $m=1$ case recovers the forest polynomials from earlier. By replacing binary forests with $(m+1)$-ary forests, there is an analogously defined set $\indexedforests[m]$ whose compositional monoid structure is analogously identified with the $m$-Thompson monoid $$\indexedforests[m]\cong \Th[m]\coloneqq\langle \tope[m]{1},\tope[m]{2},\ldots\suchthat \tope[m]{i}\tope[m]{j}=\tope[m]{j}\tope[m]{i+m}\text{ for }i>m\rangle.$$
All of the combinatorics and constructions stated specifically for $\indexedforests$ carries over with minor modifications.

In the terminology of the present paper, the $m$-forest polynomials $\{\forestpoly[m]{F}\suchthat F\in\indexedforests[m]\}$ are the unique family of polynomials dual to the \newddpair $(\tope[m]{},\indexedforests[m])$ 
given by \emph{m-quasisymmetric divided differences}
$$\tope[m]{i}=\frac{\rope{i+1}^m-\rope{i}^m}{x_i}.$$
These polynomials were shown to exist in \cite[Appendix]{NST_1} by a  laborious explicit computation. 

Like before, \cite[Definition 3.5]{NST_1} guarantees a code map for $\Th[m]$ in the sense of \Cref{def:codemap}. 
Thus to show that $m$-forest polynomials exist, it suffices to find creation operators. 
This is a straightforward adaptation of the proof for $m=1$.
Let's define $\mzope{m}=1+\rope{1}^m+\rope{1}^{2m}+\cdots:\poly^+\to \poly^+$.
\begin{thm}
    We have $\sum_{i\ge 1}\mzope{m} x_i\tope{i}=\idem$ on $\poly^+$, or in other words $\mzope{m} x_i$ are creation operators for $\tope[m]{i}$. In particular, there exists a family of ``$m$-forest polynomials'' $\{\forestpoly{F}\}_{F\in \indexedforests[m]}$ dual to the \newddpair $(\tope[m]{},\Th[m])$
    with the following monomial-positive expansion
    $$\forestpoly{F}=\sum_{(i_1,\ldots,i_k)\in \Trim{F}}\mzope{m} x_{i_k}\cdots \mzope{m} x_{i_1} (1).$$
\end{thm}
We will later see an expansion in terms of ``$m$-slides'', a natural generalization of slide polynomials introduced in \cite[Section 8]{NST_1}.

%%%%%%%%%%%%%%%%%%%%%%%%%%%%%%%%%%%%%%%%%%%%%%
%%%%%%%%%%%%%%%%%%%%%%%%%%%%%%%%%%%%%%%%%%%%%%%%%%%%%%%%%%
\section{Slide polynomials and Slide expansions}
\label{sec:Slide}
%%%%%%%%%%%%%%%%%%%%%%%%%%%%%%%%%%%%%%%%%%%%%%%%%%%%%%%%%%

In this section we will show that slide polynomials are dual to a simple \newddpair. We use this to recover the slide polynomial expansions of Schubert polynomials \cite{BJS93,AssSea17} and forest polynomials \cite{NT_forest}, and to obtain a simple formula for the coefficients of the slide expansion of any $f\in\poly$.

%%%%%%%%%%%%%%%%%%%%%%%%%%%%%%
\subsection{Slide polynomials}
%%%%%%%%%%%%%%%%%%%%%%%%%%%%%%

For a sequence $a=(a_1,\ldots,a_k)$ with $a_i\ge 1$ we define the set of \emph{compatible sequences}
\begin{align}
    \compatible{a}=\{(i_1\le \cdots \le i_k):i_j\le a_j,\text{ and if }a_j<a_{j+1}\text{ then }i_j<i_{j+1}\}.
\end{align}
Note that this convention is the opposite of what the authors employed in \cite{NST_1}. 
As we shall soon see, this convention arises naturally from the new \newddpair we will shortly create.

We define the \emph{slide polynomial} to be
\begin{equation*}
    \slide{a}=\sum_{(i_1,\dots,i_k)\in \compatible{a}}x_{i_1}\cdots x_{i_k}.
\end{equation*}
\begin{eg}\label{ex:143}
    For $\slide{(1,4,3)}$ we have $\compatible{a}=\{(1,2,2),(1,2,3),(1,3,3)\}$, so
    $$\slide{(1,4,3)}=x_1x_2^2+x_1x_2x_3+x_1x_3^2.$$
\end{eg}

 Let $\winc=\{(a_1\le \cdots \le a_k)\suchthat a_i\geq 1\text{ for } 1\leq i\leq k\} $. For a sequence $a$, we define $\overline{a}\in \winc$ be the (component-wise) maximal element of $\compatible{a}$, and undefined if $\compatible{a}$ is empty. 
Then it is easily checked that $\slide{a}=\slide{\overline{a}}$ if $\overline{a}$ is defined,  and $\slide{a}=0$ otherwise.
For instance note that for $a=(1,4,3)$ in Example~\ref{ex:143} we have $\overline{a}=(1,3,3)$.
The combinatorial construction of $\overline{a}$ from $a$ is already present in \cite[Lemma 8]{ReiShi95}, see also~\cite{NT_flagged}.
As shown by Assaf and Searles, the slides $\{\slide{a}\suchthat a\in \winc\}$ form a basis of $\poly$ \cite[Theorem 3.9]{AssSea17}. 
Note that the slides \emph{ibid.} are indexed by $c\in\nvect$, via the bijection with $\winc$ given by letting $c_j$ be the number of indices $i$ such that $a_i=j$.

%%%%%%%%%%%%%%%%%%%%%%%%%%%%%%%%%%%%%%%%%%%
\subsection{Slide extractors and creators}
\label{subsec:slide_extractors_and_creators}
%%%%%%%%%%%%%%%%%%%%%%%%%%%%%%%%%%%%%%%%%%%

We define a partial monoid structure on $\winc$ by
$$(a_1,\cdots a_k)\cdot (b_1,\ldots,b_\ell)=\begin{cases}(a_1,\ldots,a_k,b_1,\ldots,b_\ell)& \text{if }a_k\le b_1\\\text{undefined}&\text{otherwise.}\end{cases}$$

This makes $\winc$ into a graded right-cancelative monoid  with $\last{(b_1,\ldots,b_k)}=\{b_k\}$ and $\fac{(b_1,\ldots,b_k)}=\{(b_1,\ldots,b_k)\}$.

Let $\rope{i}^{\infty}$ be the \emph{truncation}  operator defined by $\rope{i}^{\infty}(f)=f(x_1,\ldots,x_{i-1},x_i,0,0,\ldots)$. It is the limit of $\rope{i}^{m}(f)$ when $m$ tends to infinity, as these polynomials clearly become stable equal to $\rope{i}^{\infty}(f)$.

\begin{defn}[Slide extractor]
\label{defn:Di}
 Define the \emph{slide extractor} to be
$$\dope{i}=\rope{i+1}^{\infty}\partial_i,$$
which for $f\in \poly$ is given concretely by
$$\dope{i}f=\frac{f(x_1,\ldots,x_{i-1},x_i,0,0,\ldots)-f(x_1,\ldots,x_{i-1},0,x_i,0,\ldots)}{x_i}.$$
\end{defn}

We have $\dope{j}f\in \poly_j$, thus $\partial_i\dope{j}=0$ if $i>j$, and so $\dope{i}\dope{j}=0$. Thus the operators $\dope{i}$ give a representation of $\winc$, and with $\dope{}=\dope{1}$ we have a \newddpair $(\dope{},\winc)$.

\begin{thm}
\label{thm:Di_and_slides}
 Slide polynomials $(\slide{a})_{a\in \winc}$ form the unique dual family of polynomials to the dd-pair $(\dope{},\winc)$. 
   Thus for $(b_1\le \cdots \le b_k)\in \winc$, we have
    \begin{align*}
    \dope{i}\,\slide{b_1,\ldots,b_k}=\delta_{i,b_k}\slide{b_{1},\ldots,b_{k-1}}.
    \end{align*}
\end{thm}

Note that the formula above can be checked directly by a simple computation, as we have an explicit expansion for slide polynomials. We will instead use \Cref{thm:creation_plus_code_equal_magic}, and this will come as a consequence. 

\begin{defn}
\label{defn:slide_creator}
    Define a linear map $\builder{i}\in \End(\poly)$ as $$\builder{i}=\sum_{1\le k \le i}x_{k}\rope{k}^{i-k}\rope{i+1}^{\infty}.$$ Explicitly, $\builder{i}$ vanishes outside of $\poly_i$ and is defined on monomials of $\poly_i$ by
     $$\builder{i}(x_1^{p_1}\cdots x_j^{p_j}x_i^p)=x_1^{p_1}\cdots x_j^{p_j}(\sum_{j<k\le i}x_{k}^{p+1})$$ where  $p_j>0$ or $j=0$.
\end{defn}

\begin{prop}
\label{prop:Bi_slide_creators}
The $\builder{i}$ are creation operators for $\dope{i}$: on $\poly^+$, we have $$\sum_{i\ge 1}\builder{i}\dope{i}=\mathrm{id}.$$
\end{prop}
\begin{proof}
On the one hand, since $\rope{1}^{\infty}=\ct$ vanishes on $\poly^+$ we obtain by telescoping 
    \begin{align}
    \label{eq:Rinfty_telescope}
        \sum_{r\geq 1} (\rope{r+1}^{\infty}-\rope{r}^{\infty})=\mathrm{id}.
    \end{align}
Now, we compute that
    \begin{align*}
    (\rope{r+1}^{\infty}-\rope{r}^{\infty})f=&f(x_1,\ldots,x_{r},0,\ldots)-f(x_1,\ldots,x_{r-1},0,\ldots)\\
    =&\sum_{j\ge 0}f(x_1,\ldots,x_{r-1},0^j,x_{r},0,\ldots)-f(x_1,\ldots,x_{r-1},0^{j+1},x_{r},0,\ldots)\\
    =&\sum_{j\ge 0}\rope{r}^j\,\rope{r+j+1}^{\infty}(x_{r+j}-x_{r+j+1})\,\partial_{r+j}f\\
    =&\sum_{j\ge 0}x_r\rope{r}^j\,\rope{r+j+1}^{\infty}\partial_{r+j}f\\
    =&\sum_{j \ge 0}(x_r\rope{r}^j\,\rope{r+j+1}^{\infty})\,\dope{r+j}f.\end{align*}
    Summing this over all $r$, the coefficient of $\dope{i}f$ is then $\sum_{1 \le k \le i}x_k\rope{k}^{i-k}\rope{i+1}^{\infty}=\builder{i}$.
\end{proof}
Our next result, \Cref{prop:Bi_build_slides}, applied to increasing sequences $1\le a_1\le \cdots \le a_k$ implies that the slide polynomials are the dual family of polynomials to $(\dope{},\winc)$. 
We note that although we could have taken an alternate choice of creation operators such as $\widetilde{\builder{i}}=\sum_{1\le k \le i}x_{k}\rope{k}^{i-k}$ (because $\rope{i+1}^{\infty}\dope{i}=\dope{i}$), \Cref{prop:Bi_build_slides} shows surprisingly that composites of the $\builder{i}$ operators construct slide polynomials even for non-decreasing sequences \textemdash{} a property not formally guaranteed by the slide polynomials being the dual family to $(\dope{},\winc)$. 
This additional property of $\builder{i}$ will be needed later in \Cref{prop:slide_exp} to recover the slide expansions of Schubert and forest polynomials.

%%%%%%%%%%%
\begin{prop}
\label{prop:Bi_build_slides}
    For any sequence $(a_1,\ldots,a_k)$ with $a_i\ge 1$ we have
    $$\slide{a_1,\ldots,a_k}=\builder{a_k}\cdots \builder{a_1}(1).$$
\end{prop}
\begin{proof}
By induction, it is enough to show that if $a=(a_1,\dots,a_k)$ then $\builder{p}\slide{a}=\slide{a_1,\dots,a_k,p}$ for any $p\geq 1$. 
In what follows we write $\lambda^\ell$ for the length $\ell$ sequence $\lambda,\ldots,\lambda$. For $(i_1,\ldots,i_k)\in \compatible{a}$, we define a set $$A_{(i_1,\ldots,i_k)}=\begin{cases}\emptyset & \text{if }i_k>p\\\{(i_1,\ldots,i_\ell,i^{k-\ell+1})\suchthat i_\ell<i\le p\}&\text{if }(i_1,\ldots,i_k)=(i_1,\ldots,i_\ell,p^{k-\ell})\text{ with }i_\ell<p.\end{cases}$$
Then by definition of $\builder{p}$ and the slide polynomials as generating functions, it suffices to show that $$\bigsqcup_{(i_1,\ldots,i_k)\in \compatible{a}} A_{(i_1,\ldots,i_k)}=\compatible{a_1,\ldots,a_k,p}.$$
Firstly, the $A_{(i_1,\ldots,i_k)}$ are obviously disjoint sets, the elements being uniquely determined by the longest initial subsequence of $(i_1,\ldots,i_k)$ strictly less than $p$, so the union is disjoint as claimed. Next, we show $A_{(i_1,\ldots,i_k)}\subset \compatible{a_1,\ldots,a_k,p}$. Indeed, since $(i_1,\ldots,i_\ell,p^{k-\ell})\in \compatible{a}$,
\begin{itemize}
    \item if $\ell<k$ we must have $a_k\ge p$ and so $(i_1,\ldots,i_\ell,p^{k-\ell+1})\in \compatible{a_1,\ldots,a_k,p}$, and
    \item if $\ell=k$ then because $p>i_\ell$ we also have $(i_1,\ldots,i_\ell,p^{k-\ell+1})\in \compatible{a_1,\ldots,a_k,p}$.
\end{itemize} 
The other sequences $(i_1,\ldots,i_\ell,i^{k-\ell+1})\in A_{(i_1,\ldots,i_k)}$ must lie in $\compatible{a_1,\ldots,a_k,p}$ as well since it is a smaller sequence with the same indices at which strict ascents occur.

Finally, every sequence in $\compatible{a_1,\ldots,a_k,p}$ can be written as $(i_1,\ldots,i_\ell,i^{k-\ell+1})$ for some $0\le \ell\le k$ and $i_\ell<i\le p$, and we claim that $(i_1,\ldots,i_\ell,p^{k-\ell})\in \compatible{a}$. 
Note that because the last $k-\ell+1$ elements of $(i_1,\ldots,i_\ell,i^{k-\ell+1})$ are equal, we have $a_{k-\ell}\ge a_{k-\ell+1}\ge\cdots \ge  a_k\ge p$. Therefore as $(i_1,\ldots,i_\ell,p^{k-\ell+1})$ has the same indices of strict ascents as $(i_1,\ldots,i_\ell,i^{k-\ell+1})$ we have the sequence $(i_1,\ldots,i_\ell,p^{k-\ell+1})\in \compatible{a_1,\ldots,a_k,p}$, which in particular implies that $(i_1,\ldots,i_\ell,p^{k-\ell})\in \compatible{a}$.
\end{proof}
%%%%%%%%%%%

We can now prove \Cref{thm:Di_and_slides}.
%%%%%%%%%%%
\begin{proof}[Proof of \Cref{thm:Di_and_slides}]
    We have the code map $c:\winc\to \nvect$ given by $c(a_1\le \cdots \le a_k)=(c_1,c_2,\ldots)$ where $c_i=\#\{j\suchthat a_j=i\}$. It satisfies the conditions of \Cref{def:codemap}. The $\builder{i}$ are shown to be creation operators for $\dope{}$ in~\Cref{prop:Bi_slide_creators}. We can thus apply \Cref{thm:creation_plus_code_equal_magic}, which gives us that the dual family to $(\dope{},\winc)$ is unique, forms a basis of $\poly$, and is given explicitly by $\builder{a_k}\cdots \builder{a_1}(1)$ for $(a_1,\ldots,a_k)\in\winc$. These are precisely the slide polynomials by \Cref{prop:Bi_build_slides}, which concludes the proof.
\end{proof}

%%%%%%%%%%%%%%%%%%%%%%
\subsection{Applications}
\label{subsec:Di_applications}
%%%%%%%%%%%%%%%%%%%%%%

We first show how to recover the slide expansions of Schubert polynomials and forest polynomials, the first one being the celebrated BJS formula~\cite{BJS93}.

\begin{prop}
\label{prop:slide_exp}
     We have the following expansions for any $w\in S_\infty$ and any $F\in\indexedforests$.
    \begin{align*}\schub{w}=&\sum_{(i_1,\ldots,i_k)\in \red{w}}\slide{i_1,\ldots,i_k}\\\forestpoly{F}=&\sum_{(i_1,\ldots,i_k)\in \Trim{F}}\slide{i_1,\ldots,i_k}.
    \end{align*}
\end{prop}
\begin{proof}
Note that $\builder{i}\dope{i}=\builder{i}\rope{i+1}^{\infty}\partial_i$. Because $\tope{i}=\rope{i}\partial_i=\rope{i+1}\partial_i$, we can either absorb all or all but one $\rope{i+1}$ into $\builder{i}$ to obtain $$\builder{i}\dope{i}=\builder{i}\tope{i}=\builder{i}\partial_i.$$ Then \Cref{prop:Bi_slide_creators} shows that $\builder{i}$ are creation operators for $\partial_i$ and for $\tope{i}$. We can then use~\Cref{thm:creation_plus_code_equal_magic} for the corresponding dd-pairs:
     \begin{align*}\schub{w}=&\sum_{(i_1,\ldots,i_k)\in \red{w}}\builder{i_k}\cdots \builder{i_1}(1)=\sum_{(i_1,\ldots,i_k)\in \red{w}}\slide{i_1,\ldots,i_k}\\\forestpoly{F}=&\sum_{(i_1,\ldots,i_k)\in \Trim{F}}\builder{i_k}\cdots \builder{i_1}(1)=\sum_{(i_1,\ldots,i_k)\in \Trim{F}}\slide{i_1,\ldots,i_k}.\qedhere
    \end{align*}
\end{proof}

Because slide polynomials are a basis of $\poly$, \Cref{prop:dualpolyproperties} implies the following.
%%%%%%%%%%%%%%
\begin{cor}\label{cor:slide_expansion_easy}
     The slide expansion of a degree $k$ homogenous polynomial $f\in \poly$ is given by
     \[
        f=\sum_{(i_1\le \cdots \le i_k)\in \winc} (\dope{i_1}\cdots\dope{i_k}f) \, \slide{i_1,\dots,i_k}.
     \]
\end{cor}
%%%%%%%%%%%%%%

\begin{eg}\label{eg:doping}
Consider $f=\schub{21534}=x_1x_3^2 + x_1x_2x_3 + x_1^2x_3 + x_1x_2^2 + x_1^2x_2 + x_1^3$.
Figure~\ref{fig:dope_operators} shows applications of slide extractors in weakly decreasing order of the indices. Corollary~\ref{cor:slide_expansion_easy} says
\[
    \schub{21534}=\slide{1,3,3}+\slide{1,1,3}+\slide{1,1,1}.
\]
 \begin{figure}[!ht]
    \centering
    \includegraphics[scale=0.7]{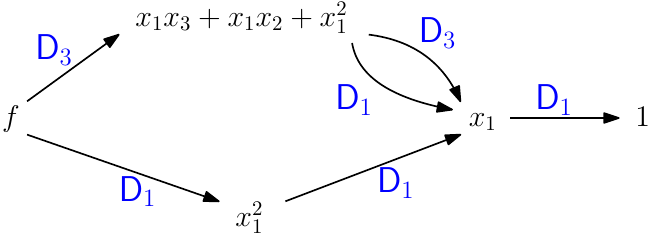}
    \caption{Repeatedly applying $\dope{}$s to extract slide coefficients for $f=\schub{21534}$}
    \label{fig:dope_operators}
\end{figure}   
\end{eg}
%%%%%%%%%%%%%%
As an application let us reprove the positivity of slide multiplication established combinatorially by Assaf--Searles \cite[Theorem 5.1]{AssSea17} using the ``quasi-shuffle product''.
In contrast we use a Leibniz rule for the $\dope{i}$ that makes the positivity manifest. 
We shall not pursue unwinding our approach to make the combinatorics explicit.
\begin{lem}
\label{lem:Rislidepos}
    $\rope{j}\slide{a}$ is a slide polynomial or $0$. 
\end{lem}
\begin{proof}
Assume the result is true for all lower degree slide polynomials. 
By \Cref{thm:Di_and_slides}, it suffices to show that $\dope{i}\rope{j}\slide{a}=0$ for all $i$, except at most one for which $\dope{i}\rope{j}\slide{a}=\slide{b}$ for some $b\in \mathsf{Winc}$. 

Let $a=(a_1,\ldots,a_k)\in \winc$, and let $a'=(a_1,\ldots,a_{k-1})\in \winc$. 
The identity $$\dope{i}\rope{j}=\begin{cases}\dope{i}&\text{if }i\le j-2\\
\dope{i}+\rope{j-1}\dope{i+1}&\text{if }i=j-1\\
\rope{j}\dope{i+1}&\text{if }i\ge j,\end{cases}$$
together with \Cref{thm:Di_and_slides} implies that $$\dope{i}\rope{j}\slide{a}=\begin{cases}\delta_{i,a_k}\slide{a'}&\text{if }a_k\le j-1\\\delta_{i,a_k-1}\rope{j-1}\slide{a'}&\text{if }a_k=j\\\delta_{i,a_k-1}\rope{j}\slide{a'}&\text{if }a_k\ge j+1,\end{cases}$$ and we conclude by the inductive hypothesis.
\end{proof}

\begin{cor}
\label{cor:slide_mult_easy}
 The product of slide polynomials is slide-positive.   
\end{cor}
\begin{proof}
By \Cref{cor:slide_expansion_easy} it suffices to show that $\dope{i}(fg)$ is slide positive if each of $f,g$ are slide positive.
    For $f,g\in \poly$ we have a ``Leibniz rule'' that says:
    \begin{align}\label{eq:leibniz}
        \dope{i}(fg)=\dope{i}(f)\rope{i+1}^{\infty}\rope{i}(g)+\rope{i+1}^{\infty}(f)\dope{i}(g).
    \end{align}
    If $f,g$ are slide polynomials, then by ~\Cref{thm:Di_and_slides} we know that $\dope{i}(f),\dope{i}(g)$ are either slide polynomials or $0$, so from \Cref{lem:Rislidepos} the slide positivity follows by induction.
    \end{proof}

Our second application is to determine the inverse of the ``Slide Kostka'' matrix, i.e. express monomials in terms of slide polynomials. 
This was obtained by the first and third author via involved combinatorial means in \cite[Theorem 5.2]{NT_flagged}.  

To state the result, fix a sequence $a=(a_1,\dots,a_k)\in \winc$. 
Group equal terms and write $a=(M_1^{m_1},M_2^{m_2},\ldots,M_p^{m_p})$, with $M_1<\cdots<M_p$. 
Set $M_0\coloneqq 0$. 
For a fixed $i\in\{1,\ldots,p\}$, define $E_i(a)\subset \winc$ by 
\[
E_i(a)=\{(b_1,\dots,b_{m_i})\suchthat b_{j+1}-b_j\in\{0,1\}=0 \text{ and } b_1>M_{i-1}\},
\] where $b_{m_i+1}\coloneqq M_i$.
Let $n(b)=M_i-b_1$ for $b\in E_i(a)$, which counts the number of $j$ such that $b_{j+1}-b_j=1$ for $1\leq j\leq m_i$.
Finally let 
\[
E(a)=\{b\in \winc\suchthat b=e^1\cdots e^{p} \text{ where each } e^i\in E_i(a)\},
\] 
To $b=e^1\cdots e^{p}\in E(a)$, assign the sign $\epsilon(b)=(-1)^{\sum_i n(e^i)}$. 
For instance, if $a=(2,4,4)$ then $E(a)=\{(2,4,4),(1,4,4),(2,3,4),(1,3,4),(2,3,3),(1,3,3)\}$
 with respective signs $1,-1,-1,1,-1,1$.
\begin{cor}[{\cite[Theorem 5.2]{NT_flagged}}]
    The slide expansion of any monomial is signed multiplicity-free. 
    Explicitly, for any $a=(a_1,\dots,a_k)\in \winc$, we have
    \begin{equation}\label{eqn:SlideKostka}
x_{a_1}\cdots x_{a_k}=\sum_{b=(j_1,\ldots,j_k)\in E(a)}\epsilon(b)\slide{b}.
    \end{equation}
\end{cor}
\begin{proof}[Sketch of the proof]
    By \Cref{cor:slide_expansion_easy} the coefficient of $\slide{j_1,\ldots,j_k}$ for $(j_1,\dots,j_k)\in \winc$ in~\eqref{eqn:SlideKostka} is given by $\dope{j_1}\dots \dope{j_k}(x_{a_1}\cdots x_{a_k})$. 
    By \Cref{defn:Di}, we can compute
    $$\dope{j_k}(x_{a_1}\cdots x_{a_k})=\begin{cases}
    x_{a_1}\cdots x_{a_{k-1}}&\text{if } a_k=j_k\\
    -x_{a_1}\ldots x_{a_p}x_{j_k}^{k-p-1}&\text{if } a_p<j_k, a_{p+1}=\cdots=a_{k}=j_k+1\text{ for some } p< k\\
    0&\text{otherwise.}\end{cases}$$
    Thus $\dope{j_k}(x_{a_1}\cdots x_{a_k})$ is either $0$ or another monomial up to sign, which shows that the expansion is signed multiplicity-free.
    More precisely, let $E'(a)$ be the set of $b=(j_1,\dots,j_k)$ such that $\slide{b}$ has nonzero coefficient in~\eqref{eqn:SlideKostka}. 
   Then it follows that $b\in E'(a)$ either if $j_k=a_k$ and $(j_1,\dots,j_{k-1})\in E'(a_1,\dots,a_{k-1})$, or if $j_k+1=a_k$, there exists $p<k$ such that $a_p<j_k$, $a_{p+1}=\dots=a_{k}=j_k+1$, and $(j_1,\dots,j_{k-1})\in E'(a_1,\dots,a_{p},j_k^{k-p-1})$. 
   We let the interested reader show that $E(a)$ satisfies the same recursion, so that $E(a)=E'(a)$ by induction. The sign is then also readily checked. 
\end{proof}

%%%%%%%%%%%%%%%%%%%%%%%%%%%%%%%%%%%%%%%%%%%%%%%%%
\subsection{$m$-slides interpolating between monomials and slides}
%%%%%%%%%%%%%%%%%%%%%%%%%%%%%%%%%%%%%%%%%%%%%%%%%%

To conclude this article, we briefly describe how the results generalize to monomials, $m$-slide polynomials and $m$-forest polynomials. The proofs are nearly identical to the case $m=1$ so we omit them.

For a sequence $a=(a_1,\ldots,a_k)$ with $a_i\ge 1$ we define the set of \emph{$m$-compatible sequences} \begin{align}
\compatible[m]{a}=\{(i_1\leq \ldots \leq i_k):i_j\equiv a_j\text{ mod }m,\, i_j\leq a_j,\text{ and if }a_j<a_{j+1}\text{ then }i_j<i_{j+1}\}.
\end{align}
The \emph{$m$-slide polynomial} \cite[Section 8]{NST_1} is the generating function
    \begin{equation}
        \slide[m]{a}=\sum_{(i_1,\dots,i_k)\in \compatible[m]{a}}x_{i_1}\cdots x_{i_k}.
    \end{equation}
For fixed $a=(a_1,\dots,a_k)$ and $m$ sufficiently large we have
$\slide[m]{a}=x_{a_1}\cdots x_{a_k}$ if $(a_1,\ldots,a_k)\in \winc$ and $0$ otherwise.
So we may consider monomials as $\infty$-slide polynomials, and the $m$-slide polynomials as interpolating between slide polynomials and monomials.

%%%%%%%%%%%
\begin{prop}
For $i\geq 1$ consider the \emph{$m$-slide extractors} $\dope[m]{i}\in \End(\poly)$ defined as
$
    \dope[m]{i}\coloneqq \rope{i+1}^{\infty}\tope[m]{i}.
$
For $(b_1\le \cdots \le b_k)\in \winc$ we have
    \[
    \dope[m]{i}\,\slide[m]{b_1,\ldots,b_k}=\delta_{i,b_k}\slide[m]{b_{1},\ldots,b_{k-1}}.
    \]
Consequentially the $m$-slide expansion of a degree $k$ homogenous polynomial $f\in \poly$ is given by
     \[
        f=\sum_{(i_1\le \cdots \le i_k)\in \winc} (\dope[m]{i_1}\cdots\dope[m]{i_k}f) \, \slide[m]{i_1,\dots,i_k}.
     \]
\end{prop}
%%%%%%%%%%%

%%%%%%%%%%%
\begin{eg}
    Taking $f=\schub{21534}=x_1x_3^2 + x_1x_2x_3 + x_1^2x_3 + x_1x_2^2 + x_1^2x_2 + x_1^3$ as in Example~\ref{eg:doping} we see for instance that 
    \[
    \dope[\infty]{1}\dope[\infty]{2}\dope[\infty]{2}(f)=\dope[\infty]{1}\dope[\infty]{2}(x_1x_2+x_1^2)=\dope[\infty]{1}(x_1)=1
    \]
    which in turns means the coefficient of $x_1x_2^2$ in $\schub{21534}$ is $1$.
\end{eg}
%%%%%%%%%%%

%%%%%%%%%%%
\begin{thm}
Consider \emph{$m$-slide creation operators} $\builder[m]{a}\in \End(\poly)$ that vanish outside of $\poly_a$, and are defined on monomials of $\poly_a$ by
     $$\builder[m]{a}(x_1^{p_1}\cdots x_j^{p_j}x_a^p)=x_1^{p_1}\cdots x_j^{p_j}(\sum_{a-rm> j}x_{a-rm}^{p+1})$$ where $j<a$ and $p_j>0$ (or $j=0$) and $p\geq 0$.
     The following hold.
    \begin{enumerate}[align=parleft,left=0pt,label=(\arabic*)]
        \item For $a=(a_1,\dots,a_k)$ be any sequence with $a_i\ge 1$, we have $\builder[m]{p}\slide[m]{a}=\slide[m]{a_1,\dots,a_k,p}.$ In particular, for any sequence $(b_1,\ldots,b_k)$ with $b_i\ge 1$ we have
    $$\slide[m]{b_1,\ldots,b_k}=\builder[m]{b_k}\cdots \builder[m]{b_1}(1).$$
        \item 
        We have $\sum_{i=1}^{\infty}\builder[m]{i}\dope[m]{i}=
    \sum_{i=1}^{\infty}\builder[m]{i}\tope[m]{i}=\idem$ on $\poly^+$, i.e. $\builder[m]{i}$ are creation operators for both $m$-slides and $m$-forest polynomials. In particular,
    \begin{align*}
    \forestpoly[m]{F}=\sum_{(i_1,\ldots,i_k)\in \Trim{F}}\slide[m]{i_1,\ldots,i_k}.\end{align*}
    \end{enumerate}
\end{thm}
%%%%%%%%%%%

%%%%%%%%%%%
\begin{rem}
    For $m=\infty$ we recover the rather straightforward \newddpair $(\dope[\infty]{},\winc)$ for monomials, where for $a_k>1$ we have $\dope[\infty]{i}(x_1^{a_1}\cdots x_k^{a_k})=\delta_{i,k}x_1^{a_1}\cdots x_k^{a_k-1}$, and the creation operators $\builder[\infty]{i}(x_1^{a_1}\cdots x_k^{a_k})=\delta_{i\ge k}x_i(x_1^{a_1}\cdots x_k^{a_k})$.  
\end{rem}
%%%%%%%%%%%

%%%%%%%%%%%%%%%%%%%%%%%%%%%%%%%%%%%%%%%%%%%%%%%%%%
\bibliographystyle{hplain}
\bibliography{schubib}

\begin{thebibliography}{10}

\bibitem{As22}
S.~Assaf.
\newblock A bijective proof of {K}ohnert's rule for {S}chubert polynomials.
\newblock {\em Comb. Theory}, 2(1):Paper No. 5, 9pp., 2022.

\bibitem{AssSea17}
S.~Assaf and D.~Searles.
\newblock Schubert polynomials, slide polynomials, {S}tanley symmetric
  functions and quasi-{Y}amanouchi pipe dreams.
\newblock {\em Adv. Math.}, 306:89--122, 2017.

\bibitem{BerBil93}
N.~Bergeron and S.~Billey.
\newblock R{C}-graphs and {S}chubert polynomials.
\newblock {\em Experiment. Math.}, 2(4):257--269, 1993.

\bibitem{BS98}
N.~Bergeron and F.~Sottile.
\newblock Schubert polynomials, the {B}ruhat order, and the geometry of flag
  manifolds.
\newblock {\em Duke Math. J.}, 95(2):373--423, 1998.

\bibitem{BJS93}
S.~C. Billey, W.~Jockusch, and R.~P. Stanley.
\newblock Some combinatorial properties of {S}chubert polynomials.
\newblock {\em J. Algebraic Combin.}, 2(4):345--374, 1993.

\bibitem{FoGrReSh1997balanced}
S.~Fomin, C.~Greene, V.~Reiner, and M.~Shimozono.
\newblock Balanced labellings and {S}chubert polynomials.
\newblock {\em European J. Combin.}, 18(4):373--389, 1997.

\bibitem{FK96}
S.~Fomin and A.~N. Kirillov.
\newblock The {Y}ang-{B}axter equation, symmetric functions, and {S}chubert
  polynomials.
\newblock In {\em Proceedings of the 5th {C}onference on {F}ormal {P}ower
  {S}eries and {A}lgebraic {C}ombinatorics ({F}lorence, 1993)}, volume 153,
  pages 123--143, 1996.

\bibitem{FS94}
S.~Fomin and R.~P. Stanley.
\newblock Schubert polynomials and the nil-{C}oxeter algebra.
\newblock {\em Adv. Math.}, 103(2):196--207, 1994.

\bibitem{hicks2024quasisymmetric}
A.~Hicks and E.~Niese.
\newblock Quasisymmetric divided difference operators and polynomial bases,
  2024, 2406.02420.

\bibitem{Hi00}
F.~Hivert.
\newblock Hecke algebras, difference operators, and quasi-symmetric functions.
\newblock {\em Adv. Math.}, 155(2):181--238, 2000.

\bibitem{Kohnert1991}
A.~Kohnert.
\newblock Weintrauben, {P}olynome, {T}ableaux.
\newblock {\em Bayreuth. Math. Schr.}, (38):1--97, 1991.
\newblock Dissertation, Universit\"{a}t Bayreuth, Bayreuth, 1990.

\bibitem{KraPra1987foncteurs}
W.~Kra\'{s}kiewicz and P.~Pragacz.
\newblock Foncteurs de {S}chubert.
\newblock {\em C. R. Acad. Sci. Paris S\'{e}r. I Math.}, 304(9):209--211, 1987.

\bibitem{KraPra2004Schubert}
W.~Kra\'{s}kiewicz and P.~Pragacz.
\newblock Schubert functors and {S}chubert polynomials.
\newblock {\em European J. Combin.}, 25(8):1327--1344, 2004.

\bibitem{LamLeeShi21}
T.~Lam, S.~J. Lee, and M.~Shimozono.
\newblock Back stable {S}chubert calculus.
\newblock {\em Compos. Math.}, 157(5):883--962, 2021.

\bibitem{LS82}
A.~Lascoux and M.-P. Sch\"{u}tzenberger.
\newblock Polyn\^{o}mes de {S}chubert.
\newblock {\em C. R. Acad. Sci. Paris S\'{e}r. I Math.}, 294(13):447--450,
  1982.

\bibitem{LS85}
A.~Lascoux and M.-P. Sch\"{u}tzenberger.
\newblock Schubert polynomials and the {L}ittlewood-{R}ichardson rule.
\newblock {\em Lett. Math. Phys.}, 10(2-3):111--124, 1985.

\bibitem{LS87}
A.~Lascoux and M.-P. Sch\"{u}tzenberger.
\newblock Symmetrization operators in polynomial rings.
\newblock {\em Funktsional. Anal. i Prilozhen.}, 21(4):77--78, 1987.
\newblock Translated from the English by A. V. Zelevinski\u{\i}.

\bibitem{Mag98}
P.~Magyar.
\newblock Schubert polynomials and {B}ott-{S}amelson varieties.
\newblock {\em Comment. Math. Helv.}, 73(4):603--636, 1998.

\bibitem{MSStD22}
K.~M\'{e}sz\'{a}ros, L.~Setiabrata, and A.~St.~Dizier.
\newblock An orthodontia formula for {G}rothendieck polynomials.
\newblock {\em Trans. Amer. Math. Soc.}, 375(2):1281--1303, 2022.

\bibitem{Mo59}
D.~Monk.
\newblock The geometry of flag manifolds.
\newblock {\em Proc. London Math. Soc. (3)}, 9:253--286, 1959.

\bibitem{NST_1}
P.~Nadeau, H.~Spink, and V.~Tewari.
\newblock Quasisymmetric divided differences, 2024, 2406.01510.

\bibitem{NT_forest}
P.~Nadeau and V.~Tewari.
\newblock Forest polynomials and the class of the permutahedral variety, 2023,
  2306.10939.

\bibitem{NT_flagged}
P.~Nadeau and V.~Tewari.
\newblock ${P}$-partitions with flags and back stable quasisymmetric functions,
  2023, 2303.09019.

\bibitem{ReiShi95}
V.~Reiner and M.~Shimozono.
\newblock Key polynomials and a flagged {L}ittlewood-{R}ichardson rule.
\newblock {\em J. Combin. Theory Ser. A}, 70(1):107--143, 1995.

\bibitem{Sot96}
F.~Sottile.
\newblock Pieri's formula for flag manifolds and {S}chubert polynomials.
\newblock {\em Ann. Inst. Fourier (Grenoble)}, 46(1):89--110, 1996.

\bibitem{WY18}
A.~Weigandt and A.~Yong.
\newblock The prism tableau model for {S}chubert polynomials.
\newblock {\em J. Combin. Theory Ser. A}, 154:551--582, 2018.

\end{thebibliography}

\end{document}